\documentclass[12pt]{amsart}
\usepackage{graphicx,amsmath,mathtools,amssymb,epsfig,color}
\usepackage{caption}
\usepackage{subcaption}
\usepackage[abs]{overpic}
\usepackage{adjustbox}
\usepackage{epstopdf,wrapfig,marvosym, hyperref}
\newtheorem{theorem}{Theorem}[section]
\newtheorem{lemma}[theorem]{Lemma}

\newtheorem{definition}[theorem]{Definition}
\newtheorem{proposition}[theorem]{Proposition}

\newtheorem{corollary}[theorem]{Corollary}

\numberwithin{equation}{section}
\newtheorem{example}[theorem]{Example}
	
	\title[Jones-Wenzl Idempotents and the M\"obius band]{Jones-Wenzl Idempotents in the Twisted $I$-bundle over the M\"obius band}
	
	\author{Dionne Ibarra}
	\address{School of Mathematics, 9 Rainforest Walk, Floor 4, Monash University, VIC 3800, Australia}
	\email{dionne.ibarra@monash.edu}
	\date{\today}
	\keywords{Jones-Wenzl idempotents,  M\"obius band, unorientable surfaces, twisted I-bundles, Kauffman bracket skein module, relative Kauffman bracket skein module}
	\subjclass[2020]{Primary: 57K10. Secondary: 57K31.}

\begin{document}
\begin{abstract}
 The Jones-Wenzl idempotent plays a vital role in quantum invariants of $3$-manifolds and the colored Jones polynomial; it also serves as a useful tool for simplifying computations and proving theorems in knot theory. The relative Kauffman bracket skein module (RKBSM) for surface $I$-bundles and manifolds with marked boundaries have a well understood algebraic structure due to the work of J. H. Przytycki and T. T. Q. L\^e. It has been well documented that the RKBSM of the $I$-bundle of the annulus and the twisted $I$-bundle over the M\"obius band have distinct algebraic structures coming from the $I$-bundle structures. This paper serves as an introduction to studying the trace of Jones-Wenzl idempotents in the Kauffman bracket skein module (KBSM) of the twisted $I$-bundle of unorientable surfaces. We will give various results on Jones-Wenzl idempotents in the KBSM of the twisted $I$-bundle over the M\"obius band when it is closed through the crosscap of the M\"obius band. We will also uncover analog properties of Jones-Wenzl idempotents in the KBSM of the twisted $I$-bundle over the M\"obius band with the preservation of the $I$-bundle structure that differ from the KBSM of $Ann \times I$.
\end{abstract}

	\maketitle

\tableofcontents

	\section{Introduction}
 The Jones-Wenzl idempotent, discovered by V. F. R. Jones in \cite{Jon1},  is an idempotent element in the Temperley-Lieb algebra. Originally, it was described as a certain symmetrizer using the Artin braid group and the projection to the Temperley-Lieb algebra.  In the late 1980's, H. Wenzl in \cite{Wen} discovered a recursive formula to the Jones-Wenzl idempotent. This formula is now widely used as the definition, see \cite{Lic}. 

The Jones-Wenzl idempotent has played a significant role in defining quantum invariants of knots and $3$-manifolds. For example, W. B. R. Lickorish's Kauffman bracket skein theoretic approach to the Witten-Reshetikhin-Turaev $3$-manifold invariants in \cite{Lic1} uses a linear combination of the trace (closure) of the idempotent elements along a framed knot or link. Similarly, the colored Jones polynomial quantum knot invariant is defined by taking the trace of the $n^{th}$ Jones-Wenzl idempotent along a $0$-framed knot in $S^3$, see \cite{Le, PBIMW}. Versions of these idempotents are also used to decorate the edges of a tetrahedron to obtain the quantum $6j$-symbols that are used in the definition of the Turaev-Viro quantum $3$-manifold invariants, see \cite{TV}.

The Jones-Wenzl idempotent has been a vital tool for simplifying computations and proving theorems in knot theory. An example of this is seen in X. Cai's proof of a closed formula for the Gram determinant of type $A$ in \cite{Cai} and a closed formula for its generalization in \cite{BIMP}. In fact, this paper was conceived by needing properties of the Jones-Wenzl idempotents when it is closed in the twisted $I$-bundle over the M\"obius band in hopes to take a similar approach to \cite{Cai} and \cite{BIMP} to prove a closed formula for the Gram determinant of type $Mb$. 

In Section \ref{IntroJW} we explain the original definition of Jones-Wenzl idempotents and also the RKBSM of the twisted $I$-bundle over the M\"obius band, then in Section \ref{MBcrossingless} we give an illustration of the two different models of the M\"obius band as well as the antipodal properties of the crosscap. In Section \ref{JWMB} we prove many corollaries to the trace of Jones-Wenzl idempotents intersecting or surrounding the crosscap, then we end with two formulas, one formula for when $n-1$ curves from $f_n$ are closed around the crosscap and the last arc is closed through the crosscap and another formula for when $n-2$ curves from $f_n$ are closed around the crosscap and the remaining two arcs intersect the crosscap. 
\subsection*{Acknowledgements}

This work was supported by the Australian Research Council grant DP210103136.



	\section{Introduction to Jones-Wenzl idempotents}\label{IntroJW}

The first formal definition of the Temperley-Lieb algebra, denoted by $\mathit{TL}_n$, was given by R. J. Baxter in \cite{Bax2} while describing the work of physicists  N. Temperley and E. Lieb in \cite{TL}.  Jones independently introduced $TL_n$ in \cite{Jon1} while working on von Neumann algebras.

\begin{definition} \label{TEMPERLEY:tln}
Let $R$ be a commutative ring with unity and $d \in R$. Let $n \in \mathbb{N}$ be fixed, then the $n^{th}$ \textbf{Temperley-Lieb algebra}, $\mathit{TL}_n$, is defined to be the unital associative algebra over R with generators $ e_1, \dots , e_{n-1}$, identity element $1_n$, 
and relations
\begin{enumerate}
    \item $e_i e_{j} e_i = e_i \text{ for } |i-j|=1$,
    \item $e_i e_j = e_j e_i \text{ for } |i-j|>1$,
    \item $e_i^2 = d e_i$.
\end{enumerate}
\end{definition}

 L. H.  Kauffman in \cite{Kau2}, motivated by utilizing the Kauffman bracket, considered the Temperley-Lieb algebra over $R=\mathbb{Z}[A^{\pm1}]$, where $A$ is an indeterminate and $d = -A^2 -A^{-2}$. He then constructed a graphical interpretation using tangles. 

We will consider an \textbf{$n$-tangle} to be a rectangular shaped disk with $n$ marked boundary points on the left (input points) and $n$ marked boundary points on the right (output points). Kauffman's graphical interpretation of the Temperley-Lieb algebra is obtained from the basis of crossingless tangles where the identity element corresponds to an $n$-tangle with $n$ parallel arcs in which each $i^{\mathit{th}}$ input point is connected to the $i^{\mathit{th}}$ output point, and each $e_i$ corresponds to an $n$-tangle
that has one input and one output cap on the $i^{\mathit{th}}$ and $i+1^{\mathit{th}}$ position as illustrated in Figure \ref{TEMPERLEY:graphical}. For simplicity we will label an arc by $n$ to denote $n$ parallel arcs as shown in Figure \ref{TEMPERLEY:fig1} and if an arc is not labelled we will assume that it is one arc unless it is attached to the idempotent element of $TL_n$.

\begin{figure}[ht]
\centering
\begin{subfigure}{.49\textwidth}
\centering
$\vcenter{\hbox{
\begin{overpic}[scale = .5]{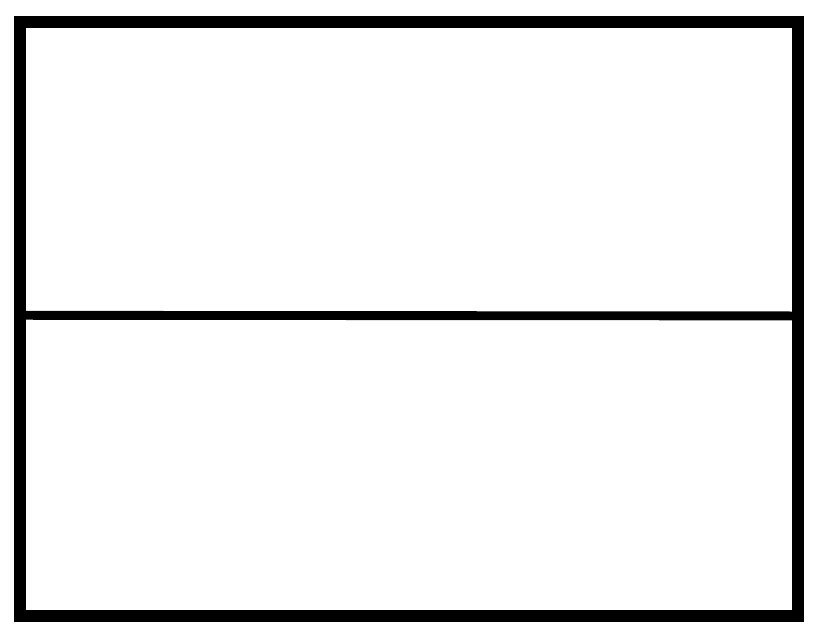}
\put(53, 48){$n$}
\end{overpic} }} $
\caption{Identity element.} \label{TEMPERLEY:fig1}
\end{subfigure}
\begin{subfigure}{.49\textwidth}
\centering
$\vcenter{\hbox{
\begin{overpic}[scale = .5]{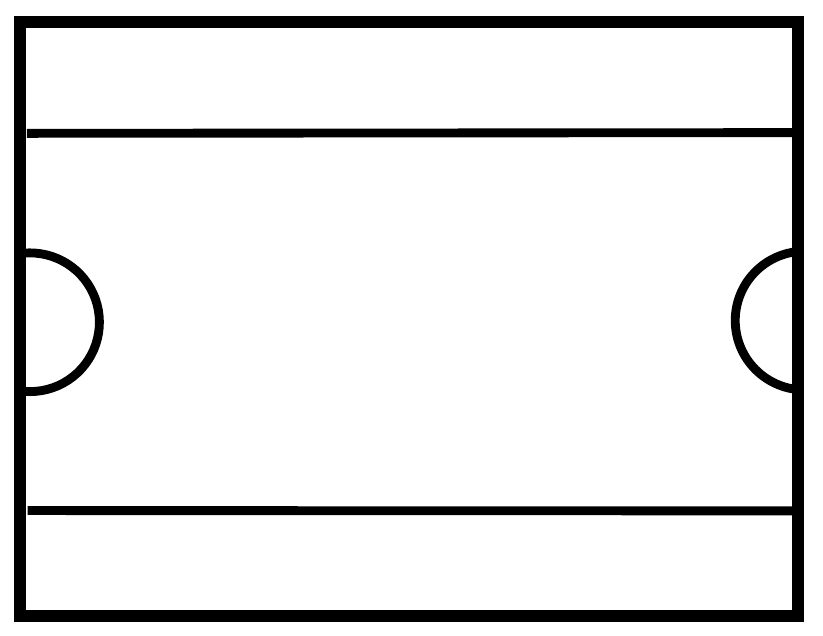}
\put(33, 20.5){$n-i-1$}
\put(45, 76){$i-1$}
\end{overpic} }} $
\caption{$e_i$.} \label{TEMPERLEY:fig2}
\end{subfigure}
\caption{The graphical interpretation of $\mathit{TL}_n$.}\label{TEMPERLEY:graphical}
\end{figure}

\begin{definition}
The \textbf{$n$-tangle algebra} is an $R$-module with basis elements consisting of $n$-tangles where multiplication of two $n$-tangles is defined by identifying the right side of the first $n$-tangle to the left side of the second $n$-tangle while respecting the boundary points and by letting any resulting trivial curve be denoted by $d$, see Figure \ref{JONESWENZL:mutip} for an illustrative example. Kauffman's diagrammatic interpretation of the Temperley-Lieb algebra, also known as the \textbf{diagrammatic algebra}, is a subalgebra of the $n$-tangle algebra. It is generated by tangles with no crossings where homotopically trivial curves are denoted by $d \in R$.  
\end{definition}

\begin{figure}[ht]
$$e_3 e_3 =\vcenter{\hbox{\begin{overpic}[scale=.5]{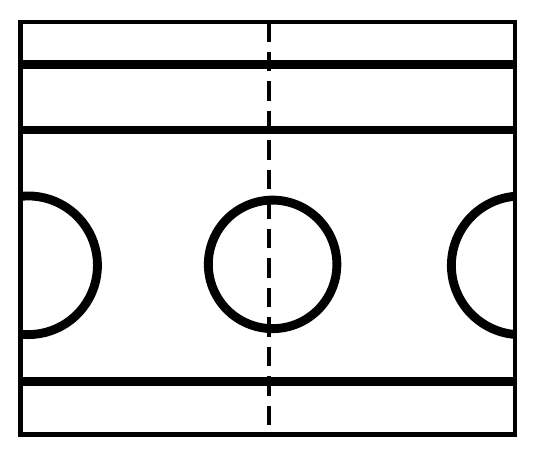}
\end{overpic}}} = d \vcenter{\hbox{\begin{overpic}[scale=.5]{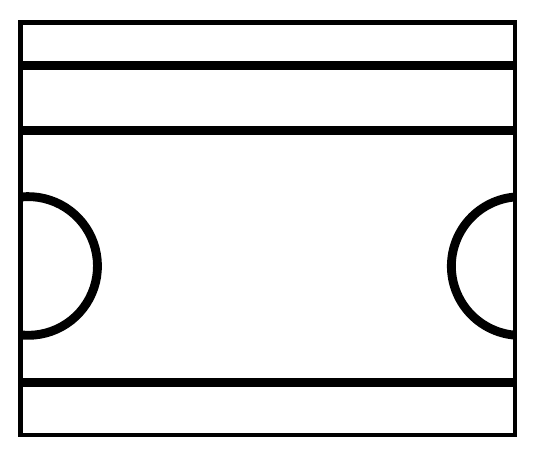}
\end{overpic}}}= d e_3.$$
\caption{An illustration of  multiplication.}\label{JONESWENZL:mutip}
\end{figure}

\begin{theorem}\cite{Kau2}
The diagrammatic algebra is isomorphic to $TL_n$ and can be thought of as a diagrammatic interpretation of it.
\end{theorem}

We will give Jones' constructive definition of the Jones-Wenzl idempotent by using the relative Kauffman bracket skein module (RKBSM) and the Artin braid group before introducing Wenzl's recursive formula. In doing so, we will first introduce the RKBSM and emphasize that when the $I$-bundle structure is preserved the RKBSM of the twisted $I$-bundle over the M\"obius band and the RKBSM of $\textit{Ann} \times I$ are different modules. 

\begin{definition}
Let $M$ be an oriented $3$-manifold and  $\{x_i\}_{i=1}^{2n}$ be the set of $2n$ framed points on $\partial M$. Let $I=[-1, 1]$, and let $\mathcal{L}^{\mathit{fr}}(2n)$ be the set of all relative framed links (which consists of all framed links in $M$ and all framed arcs, $I \times I$, where $ I \times \partial I$ is connected to framed points on the boundary of $M$) up to ambient isotopy while keeping the boundary fixed in such a way that $L \cap \partial M = \{x_i\}_1^{2n}$. Let $R$ be a commutative ring with unity,  $A \in R$ be invertible, and let $S_{2,\infty}^{\mathit{sub}}(2n)$ be the submodule of $R\mathcal{L}^{\mathit{fr}}(2n)$ that is generated by the Kauffman bracket skein relations:

\begin{itemize}
    \item [(i)]$L_+ - AL_0 - A^{-1}L_{\infty}$, and
    \item [(ii)] $L \sqcup \pmb{\bigcirc}  + (A^2 + A^{-2})L$,
\end{itemize}
where \pmb{$\bigcirc$} denotes the framed unknot and the {\it skein triple} $(L_+$, $L_0$, $L_{\infty})$ denotes three framed links in $M$ that are identical except in a small $3$-ball in $M$ where the difference is shown in Figure \ref{KBSM:skeintriple}.

Then, the \textbf{relative Kauffman bracket skein module} (RKBSM) of $M$ is the quotient: $$\mathcal{S}_{2,\infty}(M, \{x_i\}_1^{2n}; R, A) = R\mathcal{L}^{\mathit{fr}}(2n) / S_{2,\infty}^{\mathit{sub}}(2n).$$

The \textbf{Kauffman bracket skein module} (KBSM) of a manifold $M$ is defined similarly with the exception that there are no marked points and instead of considering relative framed links we restrict to framed links. If $M = F \times I$ where $F$ is an orientable surface, then we can define an algebra over the module where the identity element is the empty link $\varnothing$ and multiplication is defined as follows. Given two framed links $L_1$ and $L_2$ in $F \times I$ ($F \hat{\times} I$), their product $L_1 \cdot L_2$ is defined by placing $L_1$ over $L_2$. In particular, $L_1 \subset F\times (0, 1)$ and $L_2 \subset F\times(-1,0)$. This algebra is called the \textbf{Kauffman bracket skein algebra} and is denoted by $\mathcal S^{\mathit{alg}}(F; R,A)$.
\end{definition}
\begin{figure}[ht]
\centering
\begin{subfigure}{.25\textwidth}
\centering
$\vcenter{\hbox{\includegraphics[scale = .35]{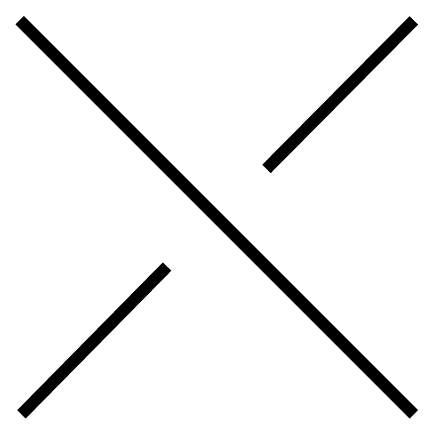}}} $
\caption{$L_+$.} \label{KBSM:Lplus}
\end{subfigure}
\centering
\begin{subfigure}{.25\textwidth}
\centering
$\vcenter{\hbox{\includegraphics[scale = .35]{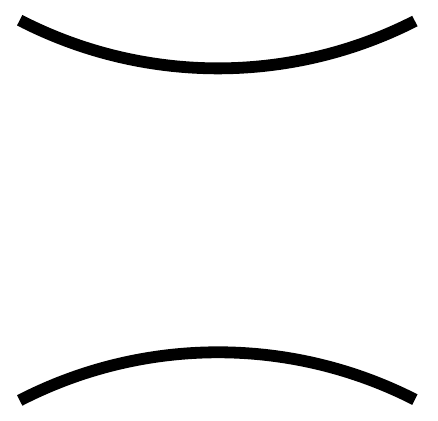}}} $
\caption{$L_0$.} \label{KBSM:Lzero}
\end{subfigure}
\centering
\begin{subfigure}{.25\textwidth}
\centering
$\vcenter{\hbox{\includegraphics[scale = .35]{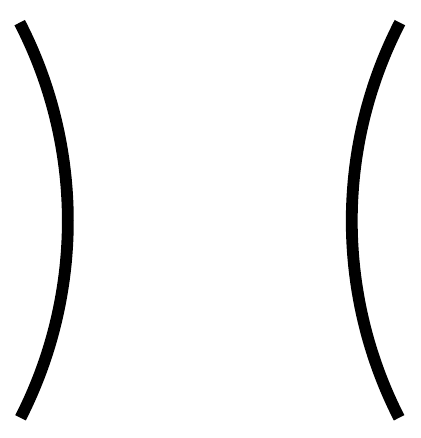}}} $
\caption{$L_{\infty}$.} \label{KBSM:Linfinity}
\end{subfigure}
\caption{The skein triple.} \label{KBSM:skeintriple}
\end{figure}

\begin{theorem}\label{rkbsmsib}\cite{Prz}
Let $F$ be a surface with $\partial F \neq \varnothing$. If $F$ is orientable then let $M = F \times \ I$ , otherwise let $M = F \hat{\times} I$. Let all $\{x_i\}_1^{2n}$  be marked points that lie on $\partial F \times \{0\}$. Then $\mathcal S_{2, \infty}(M, \{x_i\}_1^{2n}; R, A)$ is a free $R$-module whose basis is composed of relative links in $F$ without trivial components. When $n=0$, the empty link is also a generator.
\end{theorem}

J. H. Przytycki's corollary to Theorem \ref{rkbsmsib} explicitly details the differences between the basis elements of the RKBSM of $\mathit{Ann} \times I$ as an $R[z]$-module and the RKBSM of $Mb \hat{\times} I$ as an $R$-module, even though both manifolds are homeomorphic to the solid torus. This is directly due to the preservation of the $I$-bundle structure. Therefore, we are preserving the $I$-bundle structure by describing elements of $\mathcal{S}_{2,\infty}(Mb \hat\times I, \{x_i\}_1^{2n}; R, A)$ in terms of its standard basis.

\begin{corollary}\label{rkbsmcor}\cite{Prz}

\begin{enumerate}
    \item $\mathcal{S}_{2,\infty}(\mathit{Ann} \times I, \{x_i\}_1^{2n}; R, A)$ where $\{x_i\}_1^{2n}$ are located in the outer boundary component of the annulus is a free $R[z]$-module with $D_n = \binom{2n}{n}$ basis elements, where $z$ denotes the homotopically nontrivial curve in the annulus and $d=-A^2-A^{-2}$ denotes the homotopically trivial curve in the annulus. The basis is the set of all crossingless connections in the annulus with no trivial components or boundary parallel curves.  \\

    \item $\mathcal{S}_{2,\infty}(Mb \hat\times I, \{x_i\}_1^{2n}; R, A)$ is a free $R$-module. The standard basis contains an infinite number of elements of the form $bz^i$, $bxz^i$ for $i \geq 0$, where $x$ denotes the simple closed curve that intersects the crosscap once, $z$ denotes the boundary parallel curve of the M\"obius band, and $b$ is an element in the set of crossingless connections in the M\"obius band with no trivial components or boundary parallel curves for which the arcs do not intersect the crosscap. The rest of the elements in the standard basis are from a finite number of crossingless connections consisting of a collection of $n-k$ arcs for $0 \leq k <n$ that non-trivially intersect the crosscap. Among the finite collection there are $\binom{2n}{k}$ crossingless connections that intersect the crosscap $n-k$ times.  
\end{enumerate}

\end{corollary}

\begin{definition}The \textbf{Artin braid group} is defined by the following group presentation:
$$ B_n = < \sigma_1, \dots, \sigma_{n-1} ; \sigma_i \sigma_j = \sigma_j \sigma_i \text{ for } |i-j| >1,
\sigma_{i\pm1} \sigma_i \sigma_{i\pm1} = \sigma_i \sigma_{i \pm 1} \sigma_i>. $$
\end{definition}

The Artin braid group can be interpreted using $n$-tangles where elements of $B_n$ are positive braids. More precisely, an element in $B_n$ can be represented as an $n$-tangle with positive crossings such that the boundary of each arc is attached to one input and one output point and when read from left to right each generator $\sigma_i$ corresponds to the positive crossing of the $i^{\mathit{th}}$ and $i+1^{\mathit{th}}$ arcs. That is, the $i^{\mathit{th}}$ generator element $\sigma_i$ is a positive transposition of the $i^{\mathit{th}}$ and $i+1^{\mathit{th}}$ arcs.  \\

Furthermore, there exists an epimorphism $p: B_n \to S_n$ from the Artin braid group to the permutation group that uniquely interprets a braid word. Let $p$ be defined by sending generators of $B_n$, $\sigma_i$, to the transpositions in $S_n$; $s_i = (i, i+1)$ for $1 \leq i \leq n-1$. For a permutation $\pi \in S_n$, let $b_{\pi}$ denote the unique minimal positive braid word such that $p(b_{\pi}) = \pi$.

\begin{definition}\label{JONESWENZL:Asymmetrizer}
Let $\mathbb{Z}[A^{\pm1}]$ denote the ring of Laurent polynomials in the variable $A$ and $\mathbb{Q}(A)$ denote the field of rational functions in the variable $A$; whose elements are functions of the form $P/Q$ where $P, Q \in \mathbb{Z}[A^{\pm1}]$.
We define an \textbf{unnormalized $A$-symmetrizer}, $F_n \in \mathbb{Z}[A^{\pm1}]B_n$,  by the following
$$F_n = \sum_{ \pi \in S_n} (A^3)^{| \pi |} b_{\pi}, $$
and the normalized symmetrizer, also known as the \textbf{$A$-symmetrizer} and denoted by $f_n \in \mathbb{Q}(A)B_n$, by the formula
$$f_n = \frac{1}{[n]_{A^4}!} F_n, $$

where $|\pi|$ denotes the minimal length of the permutation $\pi$ written as elementary transposition generators,  $[n]_{A^4} = 1 + A^4 + A^8 + \cdots + A^{4(n-1)} = \frac{A^{4n}-1}{A^4-1}$ and
$[n]_{A^4} ! = \prod\limits_{i=1}^{n} [i]_{A^4}$. \\
$F_n$ evaluated in $\mathcal{S}_{2,\infty}(D^2 \times I, \{x_i\}_1^{2n}; R, A) $ is an element in the Temperley-Lieb algebra and the normalization is chosen so that $f_n$ is an idempotent element in $\mathit{TL}_n$. The most recognized name for this $A$-symmetrizer is the \textbf{Jones-Wenzl idempotent}. We will denote this element as a square with $n$ strands entering and $n$ strands exiting, as shown in Figure \ref{JONESWENZL:jwfig} and $f_n$ will denote the Jones-Wenzl idempotent.
\end{definition}

\begin{figure}[ht]
$$\vcenter{\hbox{\begin{overpic}[scale=.35]{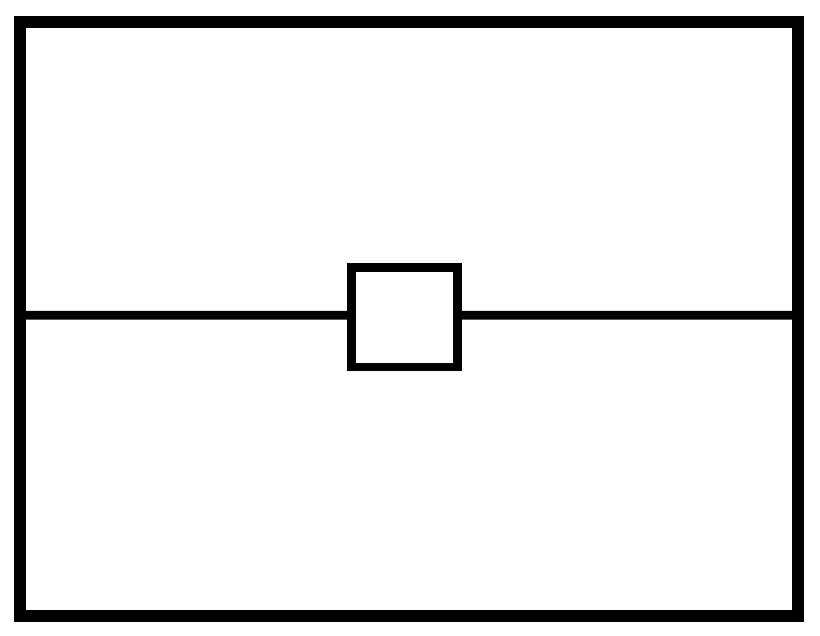}
\put(13, 35){$n$}
\end{overpic}}}$$
\caption{The Jones-Wenzl idempotent.}\label{JONESWENZL:jwfig}
\end{figure}

Wenzl's recursive formula uses the knot theoretic Chebyshev polynomial of the second kind defined below.

\begin{definition}
The $n^{th}$ \textbf{Chebyshev polynomial of the first kind} is defined recursively by the initial conditions $T_0 (d)= 2, \  T_1(d) = d$ and Equation \ref{Chebyshevfirst kindEq}.
\begin{equation} \label{Chebyshevfirst kindEq}
T_n(d) = d T_{n-1}(d) - T_{n-2}(d).
\end{equation}

The $n^{th}$ \textbf{Chebyshev polynomial of the second kind} is defined recursively by the initial conditions $S_0 (d)= 1, \  S_1(d) = d$ and the same recursive relation as the first kind, $S_n(d) = d S_{n-1}(d) - S_{n-2}(d)$. 
\end{definition}
	
When we substitute $d = -A^{2}-A^{-2}$, the Chebyshev polynomial of the first kind has the following closed formula
$$T_{n}(d) = (-1)^n(A^{2n}+A^{-2n}),$$
and the Chebyshev polynomial of the second kind has the following closed formula, denoted by $\Delta_n$,
$$\Delta_n=(-1)^n\frac{A^{2n+2}-A^{-2n-2}}{A^2-A^{-2}}=(-1)^nA^{-2n}[n+1]_{A^4}. $$

\begin{theorem}\cite{Wen} \label{JONESWENZL:wenzltheorem}
	A recursion formula for the $n^{\mathit{th}}$ Jones-Wenzl idempotent, $f_{n}$, is described in Equation \ref{JONESWENZL:wenzlrecursion}:
\begin{equation} \label{JONESWENZL:wenzlrecursion}
f_n  = 
\vcenter{\hbox{\begin{overpic}[scale=.48]{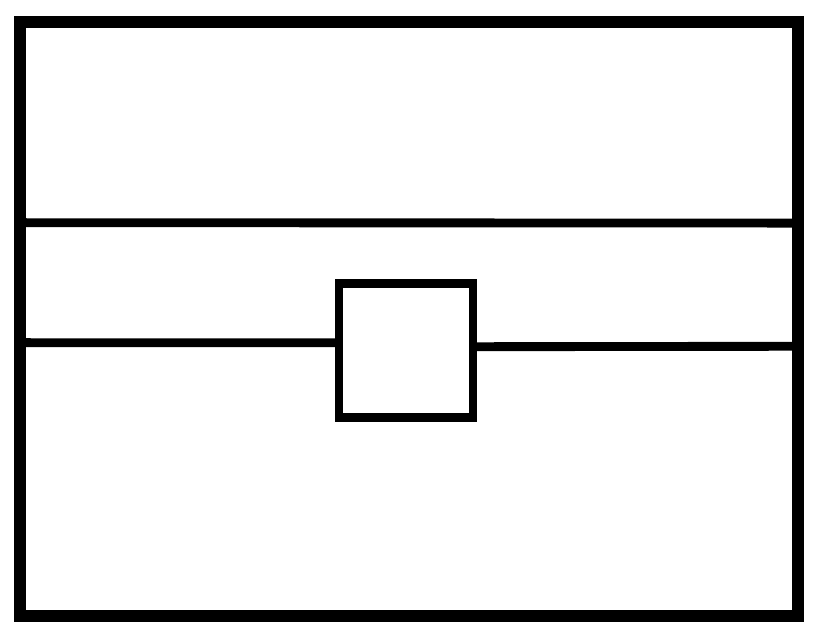}
\put(9, 44){\fontsize{7}{7}$n-1$}
\end{overpic}}} 
 - \frac{\Delta_{n-2}}{\Delta_{n-1}}
\vcenter{\hbox{\begin{overpic}[scale=.48]{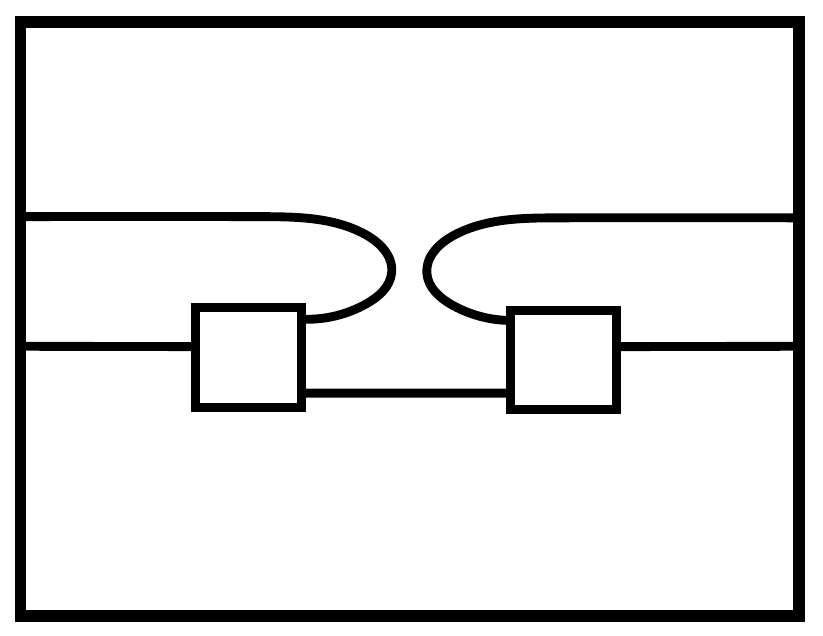}
\put(4, 42){\fontsize{6}{6}$n-1$}
\put(91, 42){\fontsize{6}{6}$n-1$}
\put(48, 36){\fontsize{6}{6}$n-2$}
\end{overpic}}}.
\end{equation}

\end{theorem}

The following lemma can be obtained from Wenzl's recursive formula as discussed in \cite{Lictln, Lic} or from the constructive definition of the Jones-Wenzl idempotent as detailed in \cite{PBIMW}.
\begin{lemma} \label{JONESWENZL:thmfnbasis}\cite{Lictln, Lic}
\begin{enumerate}
\item[(a)] $(f_n-1)$ is an element of the algebra generated by $\{ e_i\}_{i=1}^{n-1}.$
\item[(b)] $e_i f_n = f_n e_i = 0$ for $1 \leq i \leq n-1$.
\item[(c)] $f_n f_n = f_n$.
\item[(c)'] For $m\leq n$, 
$$ \vcenter{\hbox{\begin{overpic}[scale=.35]{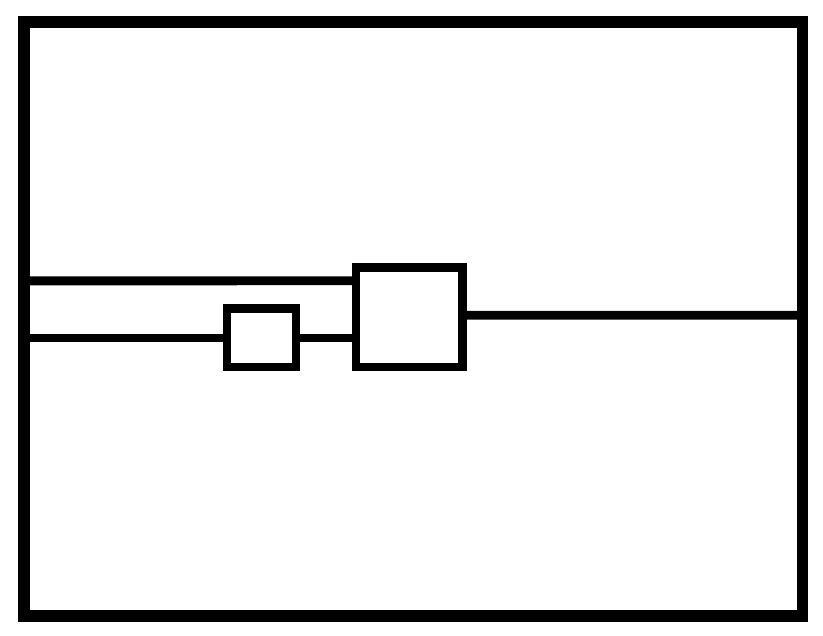}
\put(60, 35){$n$}
\put(9, 39){\fontsize{6}{6}$n-m$}
\put(9, 25){\fontsize{6}{6}$m$}
\end{overpic}}} = \vcenter{\hbox{\begin{overpic}[scale=.35]{JW1.pdf}
\put(13, 35){$n$}
\end{overpic}}}.$$
\end{enumerate}
\end{lemma}
	
A direct application of the next corollary will be given in Section \ref{JWMB}.

 \begin{corollary}\label{JONESWENZL:corollarytrace}\cite{LicCalc} Let $\mathit{tr}_1(f_n)$ be obtained from $f_n$ by closing the top string in $f_n$ (see Figure \ref{JONESWENZL:trace}). Then  $$\mathit{tr}_1(f_n)=\frac{\Delta_{n}}{\Delta_{n-1}}f_{n-1}.$$
\end{corollary}

\begin{figure}[ht] 
\centering
$\vcenter{\hbox{\begin{overpic}[scale=.4]{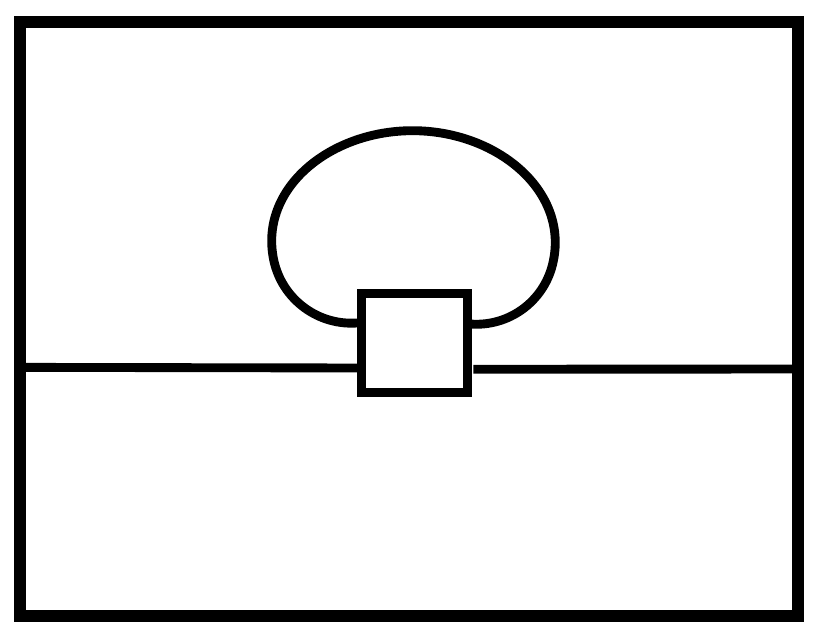}
\put(5, 34){\fontsize{9}{9}$n-1$}
\end{overpic}}} = \frac{\Delta_n}{\Delta_{n-1}}
\vcenter{\hbox{\begin{overpic}[scale=.4]{JW1.pdf}
\put(8, 39.5){\fontsize{9}{9}$n-1$}
\end{overpic}}}$
\caption{Illustration of $\mathit{tr}_1 (f_n)$.}
\label{JONESWENZL:trace}
\end{figure}

When defining the colored Jones polynomial, many authors use the term ``decorating a knot by the Chebyshev polynomial". In the case of $SU(2)$ invariants this can be described by taking the trace of $f_n$ along a framed knot (up to normalization). This is because the trace of $f_n$ along the standard annulus $S^1 \times I$, where $S^1$ is the trivial knot, is equal to the Chebyshev polynomial of the second kind as stated in the next two corollaries. 

\begin{corollary} \cite{LicCalc} Let $tr(f_n)$ denote the trace (closure) of $f_n$ obtained by identifying the left point to the right points in the trivial way, then
$$ tr(f_n) = \Delta_n.$$
\end{corollary}

\begin{corollary}\label{Cor:chebyshev}\cite{LicCalc} Let $tr_{\mathit{Ann}}(f_n)$ denote the annular trace (closure) of $f_n$ obtained by identifying the left point to the right points in the annulus and let $S_{n}(z)$ denote the $n^{th}$ Chebyshev polynomial of the second kind and $z$ denote the homotopically non-trivial curve in the annulus. Then
$$\mathit{tr}_{Ann}(f_n) = S_{n}(z). $$
\end{corollary}

\begin{lemma} \label{JONESWENZL:lemmacircle} \cite{LicCalc}
$$ \vcenter{\hbox{
\begin{overpic}[scale = .35]{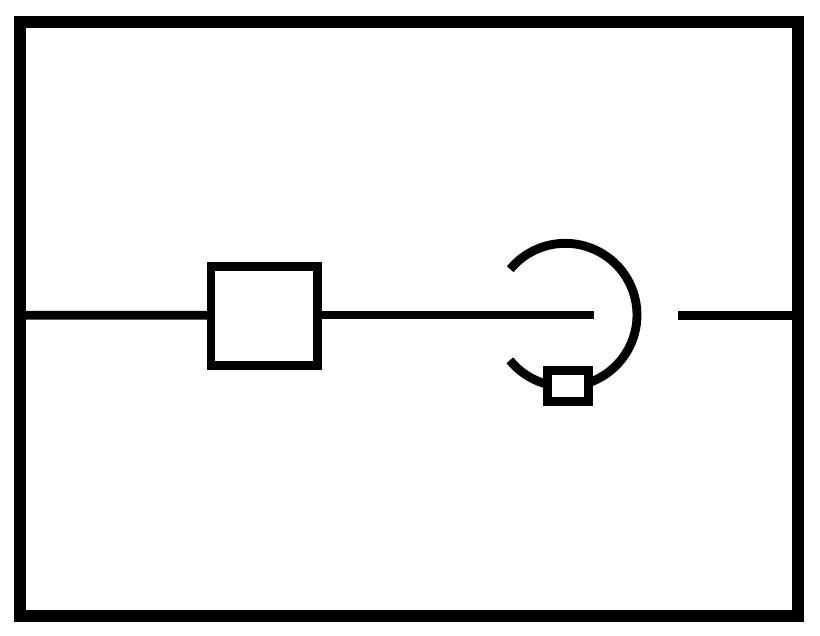}
\put(67, 35){ $b$}
\put(52, 43){ $a$}
\end{overpic} }} = \varsigma \vcenter{\hbox{
\begin{overpic}[scale = .35]{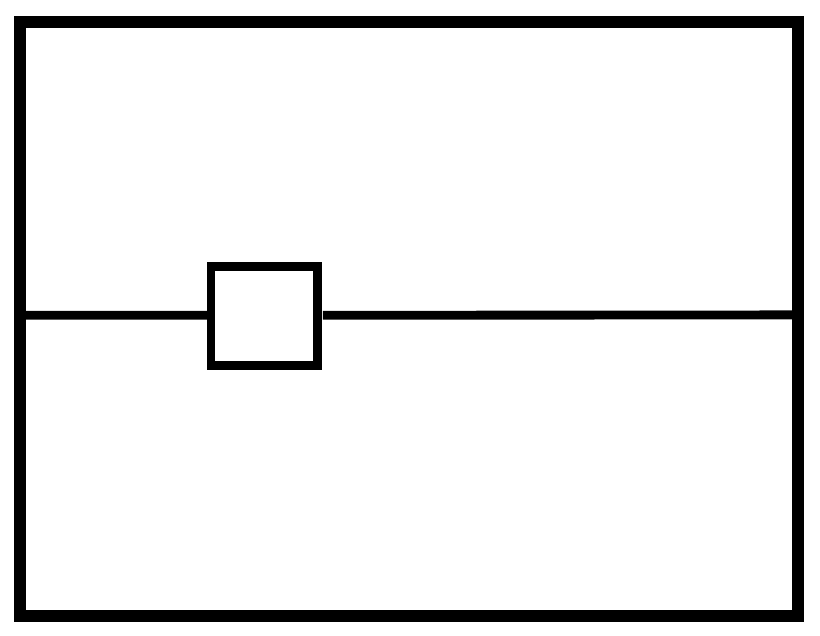}
\put(50, 35){ $b$}
\end{overpic} }},$$
where $\varsigma = \frac{(-1)^a\left( A^{2(b+1)(a+1)} - A^{-2(b+1)(a+1) }\right)}{A^{2(b+1)} - A^{-2(b+1)}}$.
\end{lemma} 

The following result is a well known corollary to Lemma \ref{JONESWENZL:lemmacircle}, we will see similar corollaries in Section \ref{JWMB} for elements in the twisted $I$-bundle over the M\"obius band. 
\begin{corollary}\cite{LicCalc}
\begin{equation}
\vcenter{\hbox{
\begin{overpic}[scale=.3]{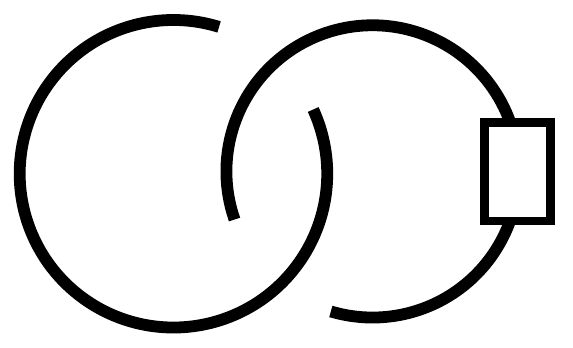}
\put(-5,0){\fontsize{9}{9}$m$}
\put(45,0){\fontsize{9}{9}$k$}
\end{overpic}}} = (-A^{2(k+1)}-A^{-2(k+1)})^m \Delta_k = ((-1)^{k}T_{k+1})^m \Delta_k,
\end{equation}
where $d = -A^2-A^{-2}$ and $T_{n}(d) = (-1)^n(A^{2n}+A^{-2n})$ is the $n^{th}$ Chebyshev polynomial of the first kind.
\end{corollary}


\section{Crossingless connection in the M\"obius band}\label{MBcrossingless}

This section serves as an overview on curves and arcs in the M\"obius band described by two different models. For more information we refer the reader to \cite{Lic2, Lic3}.
Throughout this paper we will mainly use the crosscap model of the M\"obius band where the boundary will be given in a rectangular form when marked points are included,  as shown in Figure \ref{Mobiusmodelscc2}, otherwise it will be displayed as a smooth circle as shown in Figure \ref{Mobiusmodel2}. \\

The three homotopically distinct arcs fixed on the boundary of the M\"obius band are given in Figure \ref{Mobiusmodelscc}.  In order to relate the arcs from the first and second model, a convention was chosen on the two distinct arcs fixed on the boundary that do not intersect the crosscap. 

\begin{figure}[ht] 
\centering
\begin{subfigure}{.45\textwidth}
\centering
$\vcenter{\hbox{\begin{overpic}[scale=.8]{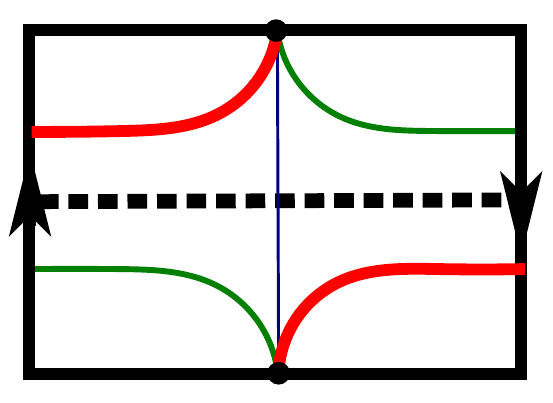}
\end{overpic}}}$
\caption{Formed from $[0,1] \times [0,1]$ by identifying $\{0\} \times [0,1]$ with $\{1\} \times [0,1]$ as shown by the arrows.}
\label{mobiusmodelscc1}
\end{subfigure} \quad
\begin{subfigure}{.45\textwidth}
\centering
$\vcenter{\hbox{\begin{overpic}[scale=.8]{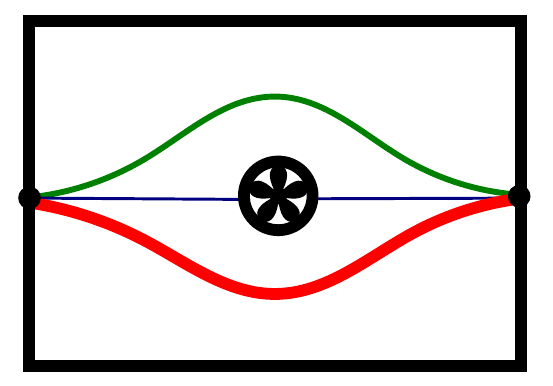}
\end{overpic}}}$
\caption{Consists of a crosscap and highlights the boundary of the M\"obius band.}
\label{Mobiusmodelscc2}
\end{subfigure}
\caption{Two models of the M\"obius band with 3 homotopically distinct arcs fixed on the boundary.}
\label{Mobiusmodelscc}
\end{figure}

\begin{figure}[ht] 
\centering
\begin{subfigure}{.45\textwidth}
\centering
$\vcenter{\hbox{\begin{overpic}[scale=.8]{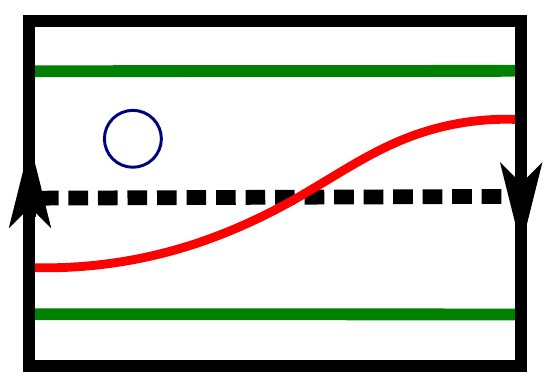}
\put(16, 55){\fontsize{9}{9}$d$}
\put(60, 77){\fontsize{9}{9}$z$}
\put(55, 30){\fontsize{9}{9}$x$}
\end{overpic}}}$
\caption{First model.}
\label{mobiusmodel1}
\end{subfigure} \quad
\begin{subfigure}{.45\textwidth}
\centering
$\vcenter{\hbox{\begin{overpic}[scale=.8]{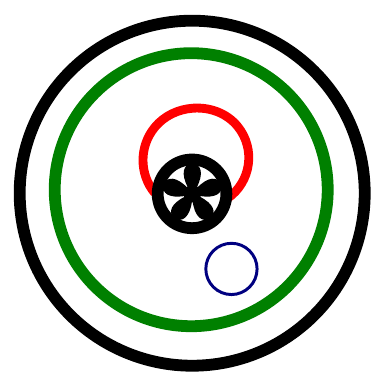}
\put(58, 32){\fontsize{9}{9}$d$}
\put(22, 60){\fontsize{9}{9}$z$}
\put(43, 57){\fontsize{9}{9}$x$}
\end{overpic}}}$
\caption{Second model.}
\label{Mobiusmodel2}
\end{subfigure}
\caption{Two models of the M\"obius band with 3 homotopically distinct simple closed curves in M\"obius band denoted by $d, x,$ and $z$, respectively.}
\label{Mobiusmodel}
\end{figure}
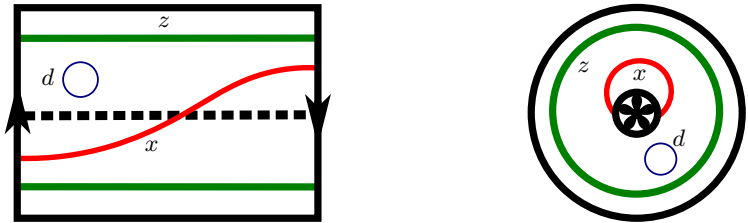

Figure \ref{Mobiusmodel} pictorially describes the three homotopically distinct simple closed curves in the M\"obius band. If a simple closed curve intersects the crosscap more than once then the number of intersection points can be reduced by two at a time. The following example will illustrate the process of removing two intersection points from the crosscap. Similar moves can be applied to arcs attached to the boundary that intersect the crosscap more than once. 

\begin{example}
We will illustrate,  in the first model then the second model,  the removal of two intersection points of the crosscap from a simple closed curve. In the two examples, the curve will be multicolored in order to show which portion of the curve passes through the crosscap. 

 Suppose we have a simple closed curve that is homotopically trivial and intersects the crosscap twice then, as shown in Equation \ref{FirstModel2ptsd}, we may use one isotopy move to remove the two intersection points.

\begin{equation}\label{FirstModel2ptsd}
\vcenter{\hbox{\begin{overpic}[scale=.4]{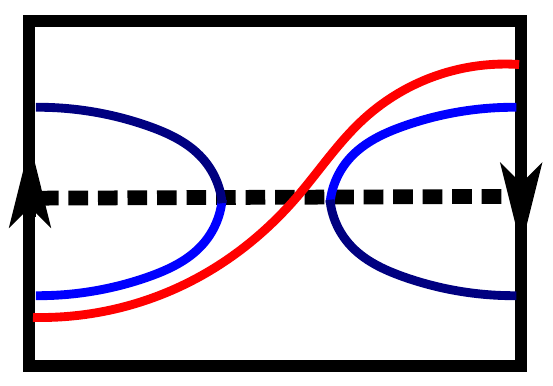}
\end{overpic}}} \sim \vcenter{\hbox{\begin{overpic}[scale=.4]{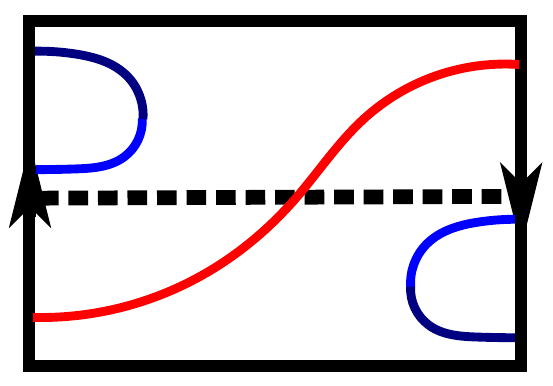}
\end{overpic}}}  \sim  \vcenter{\hbox{\begin{overpic}[scale=.4]{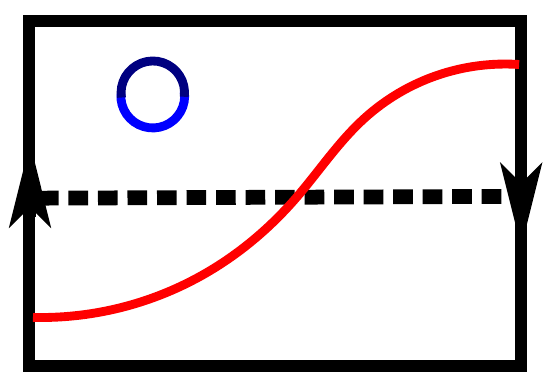} \end{overpic}}}.
\end{equation}

In Equation \ref{SecModel2ptsd}, we will illustrate the removal of the same intersection points presented in the second model.

\begin{equation}\label{SecModel2ptsd}
\vcenter{\hbox{\begin{overpic}[scale=.5]{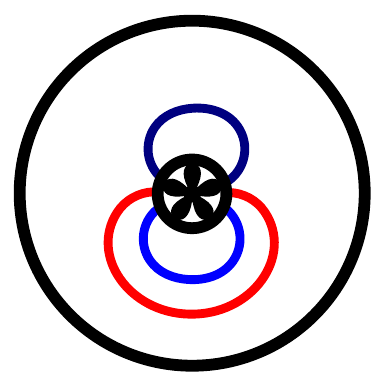}
\end{overpic}}} \sim \vcenter{\hbox{\begin{overpic}[scale=.5]{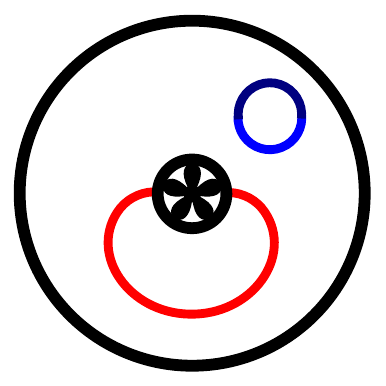}
\end{overpic}}}  .
\end{equation}

Now, suppose we have a homotopically non-trivial curve that intersects the crosscap twice, for example the curve illustrated in Equation \ref{FirstModel2ptsz}. Then, we may remove the two intersection points by using one isotopy move as given below.
\begin{equation}\label{FirstModel2ptsz}
\vcenter{\hbox{\begin{overpic}[scale=.4]{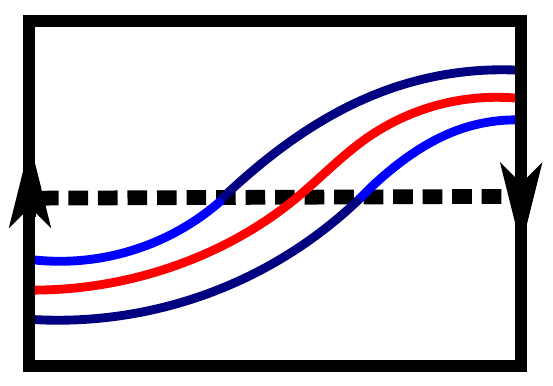}
\end{overpic}}} \sim \vcenter{\hbox{\begin{overpic}[scale=.4]{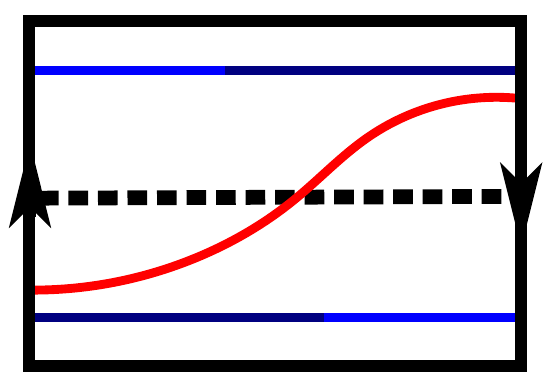}
\end{overpic}}}.
\end{equation}

Equation \ref{SecModel2ptsz} gives an illustration of this move in the second model. 

\begin{equation}\label{SecModel2ptsz}
\vcenter{\hbox{\begin{overpic}[scale=.5]{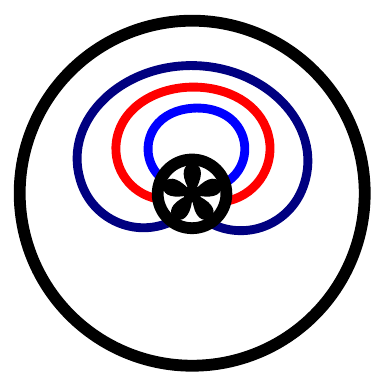}
\end{overpic}}} \sim 
\vcenter{\hbox{\begin{overpic}[scale=.5]{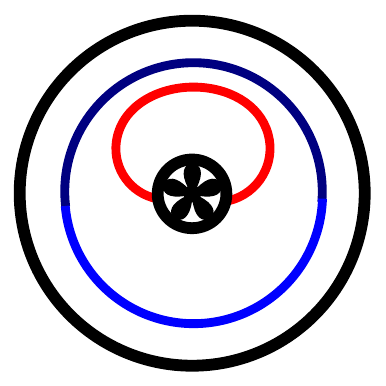}
\end{overpic}}}.
\end{equation}

\end{example}

	
	\section{Jones-Wenzl idempotents in a module on the M\"obius band}\label{JWMB}
In this section we will introduce various properties associated to the Jones-Wenzl idempotents in the twisted $I$-bundle over the M\"obius band.

\begin{figure}[ht] 
\centering
$\vcenter{\hbox{\begin{overpic}[scale=.4]{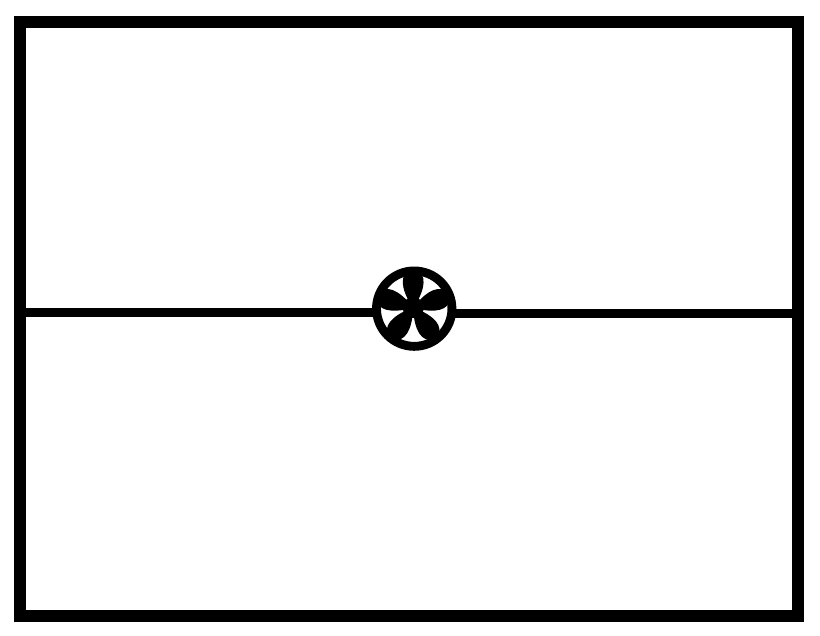}
\put(10, 39){\fontsize{9}{9}$n$}
\end{overpic}}}$
\caption{Illustration of the unique element in $Mb_n$ with $n$ arcs intersecting the crosscap, denoted by $1_{Mb_n}$.}
\label{Mobius:mn}
\end{figure}

 We start by defining a module over the Temperley-Lieb algebra generated by the basis of $\mathcal{S}_{2, \infty}(Mb \hat{\times} I, 2n)$. This module is defined in the obvious way that one would define a module over $TL_n$ generated by the basis of the RKBSM of $F \times I$ (or $F \hat{\times} I$). While trivially defined, it serves as a foundation to  describe properties of the basis $\{ e_i \}$  of $TL_n$ when juxtaposition with crossingless connections in the M\"obius band. It also serves as a first step to define an algebra from unorientable surfaces with non-empty boundary. 
 
\begin{definition}\label{TLnmodule}
Consider the basis of $\mathcal{S}_{2, \infty}(Mb \hat{\times} I, 2n)$, denoted by $\mathcal{B}(Mb \hat{\times} I, n)$ and define left multiplication 
$$\cdot : TL_n \times \mathcal{S}_{2, \infty}(Mb \hat{\times} I, 2n) \to \mathcal{S}_{2, \infty}(Mb \hat{\times} I, 2n),$$
 on the basis as follows:

For $e_i \in TL_n$ and $m \in \mathcal{B}(Mb \hat{\times} I, n)$ define $e_i \cdot m$ by a side-by-side juxtaposition where the right boundary interval of $e_i$ is identified with the left boundary interval of $m$ and the points on the intervals are trivially identified, as shown in Figure \ref{Mobius:Multiplication}. 

This yields a left Temperley-Lieb algebra-module with basis $ \mathcal{B}(Mb \hat{\times} I, n)$ called the \textbf{left $TL_n$-module over the M\"obius band}. A right $TL_n$-module over the M\"obius band is defined similarly.
\end{definition}

 Consider a subset of the basis of $\mathcal{S}_{2, \infty}(Mb \hat{\times} I, 2n)$ containing crossingless connections with no $z$ or $d$ curves, denoted by $Mb_n$. The following lemma details a relationship between the left and right  $TL_n$-module over the M\"obius band on elements of $Mb_n$.

 \begin{figure}[ht] 
\centering
$\vcenter{\hbox{\begin{overpic}[scale=.4]{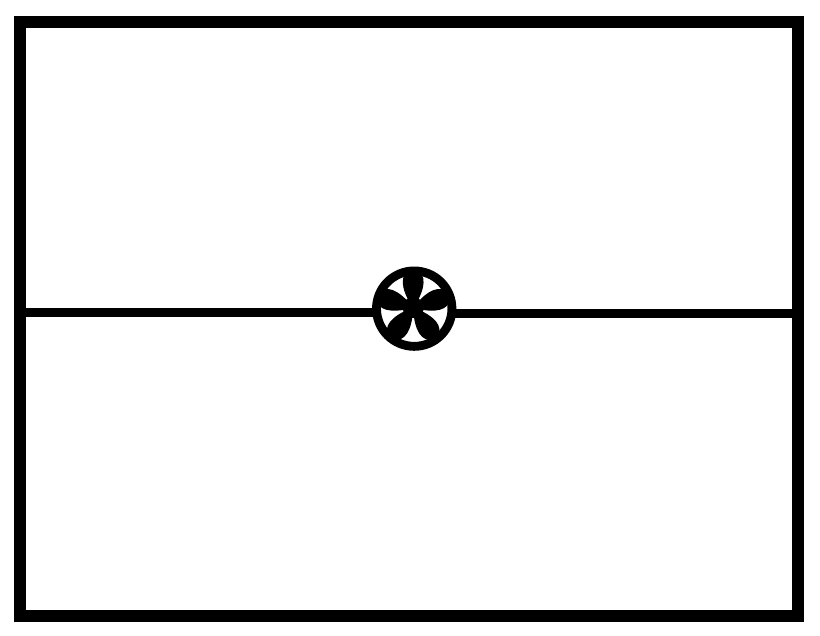}
\put(10, 39){\fontsize{9}{9}$n$}
\end{overpic}}}$.
\caption{An illustration of the unique element in $Mb_n$ with $n$ arcs intersecting the crosscap, denoted by $1_{Mb_n}$.}
\label{Mobius:mn}
\end{figure}

\begin{figure}[ht] 
\centering
$e_2 \cdot 1_{\mathit{Mb}_4} = \vcenter{\hbox{\begin{overpic}[scale=.4]{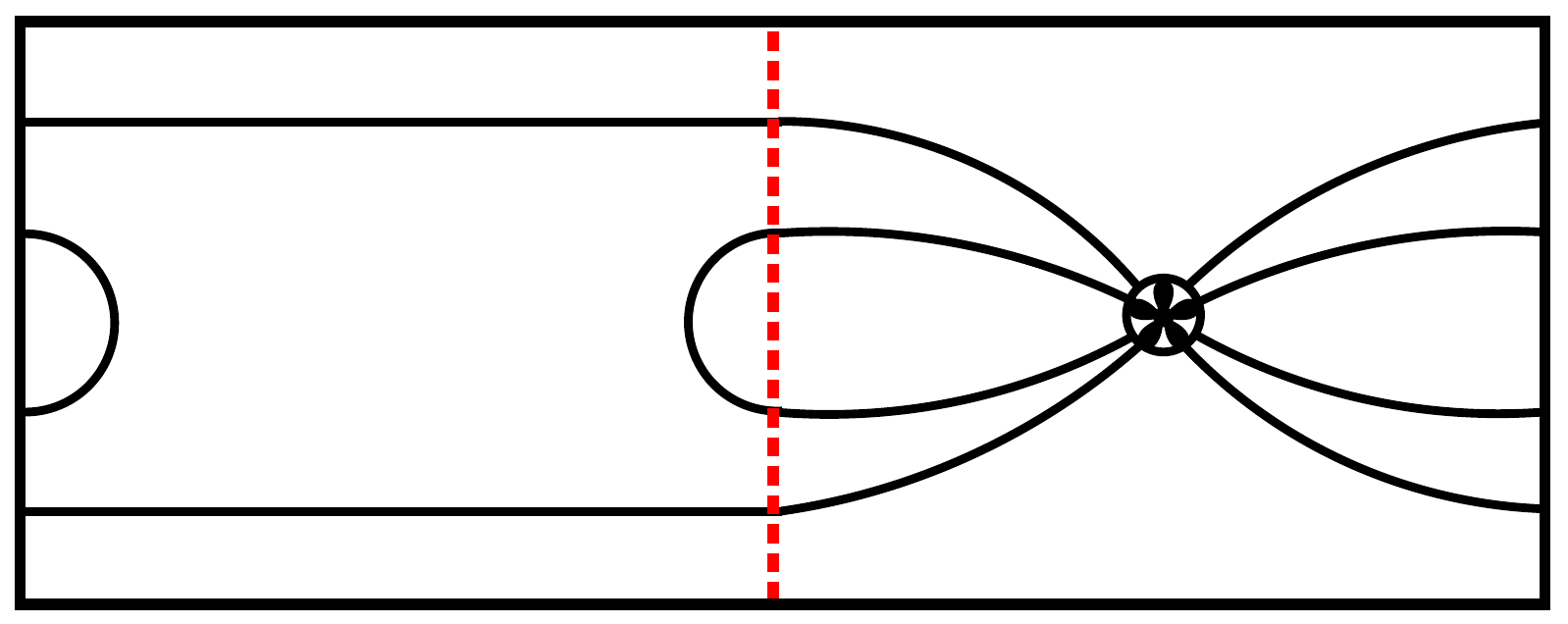}
\end{overpic}}}$.
\caption{An illustration of left multiplication.}
\label{Mobius:Multiplication}
\end{figure}

\begin{lemma}\label{Lemma:eipull} Let $(Mb_n)_k$ denote the set of elements of $Mb_n$ that intersect the crosscap $k$ times, and let $1_{Mb_n}$ denote the unique element in $Mb_n$ that intersect the crosscap $n$ times, then
$$1_{Mb_n} e_i = e_{n-i} 1_{Mb_n} \in (Mb_n)_{n-2}.$$
\end{lemma}

\begin{proof}
This is a direct result from the antipodal property of the crosscap that is explained in Section \ref{MBcrossingless} and illustrated in Equations \ref{lemma:mbei} and \ref{lemma:e1mb}.

\begin{equation}\label{lemma:mbei}
1_{Mb_n} e_i = \vcenter{\hbox{\begin{overpic}[scale=.35]{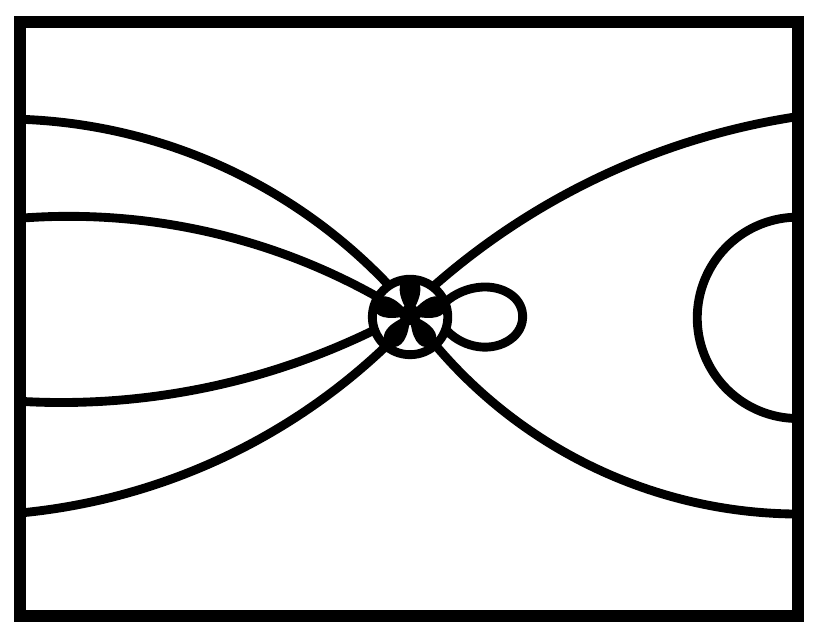}
\put(4, 54){\fontsize{8}{8}$n-i-1$}
\put(62, 54){\fontsize{8}{8}$i-1$}
\put(46, 5){\fontsize{8}{8}$n-i-1$}
\put(4, 5){\fontsize{8}{8}$i-1$}
\end{overpic}}} = \vcenter{\hbox{\begin{overpic}[scale=.35]{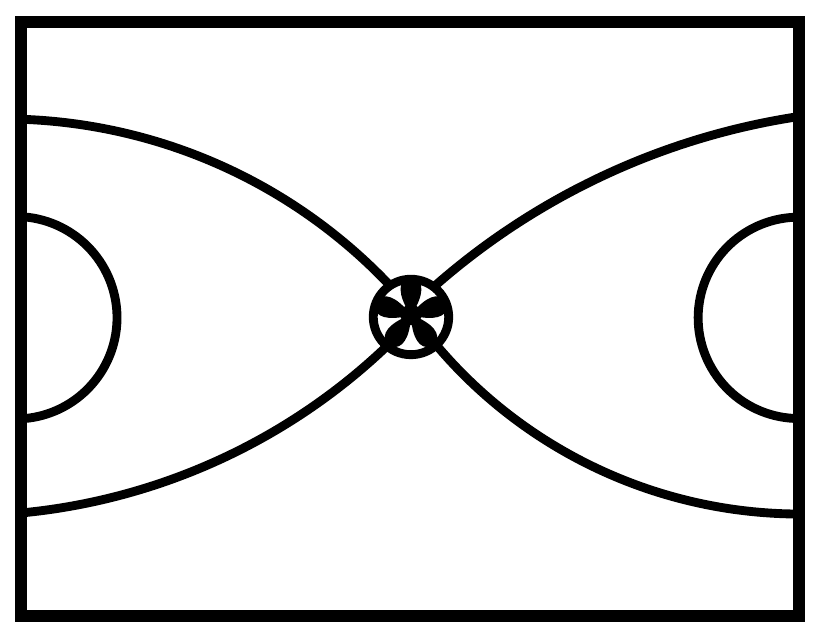}
\put(4, 54){\fontsize{8}{8}$n-i-1$}
\put(62, 54){\fontsize{8}{8}$i-1$}
\put(46, 5){\fontsize{8}{8}$n-i-1$}
\put(4, 5){\fontsize{8}{8}$i-1$}
\end{overpic}}} \in  (Mb_n)_{n-2}.
\end{equation}
\begin{equation}\label{lemma:e1mb}
e_{n-i} 1_{Mb_n}  = \vcenter{\hbox{\begin{overpic}[scale=.35]{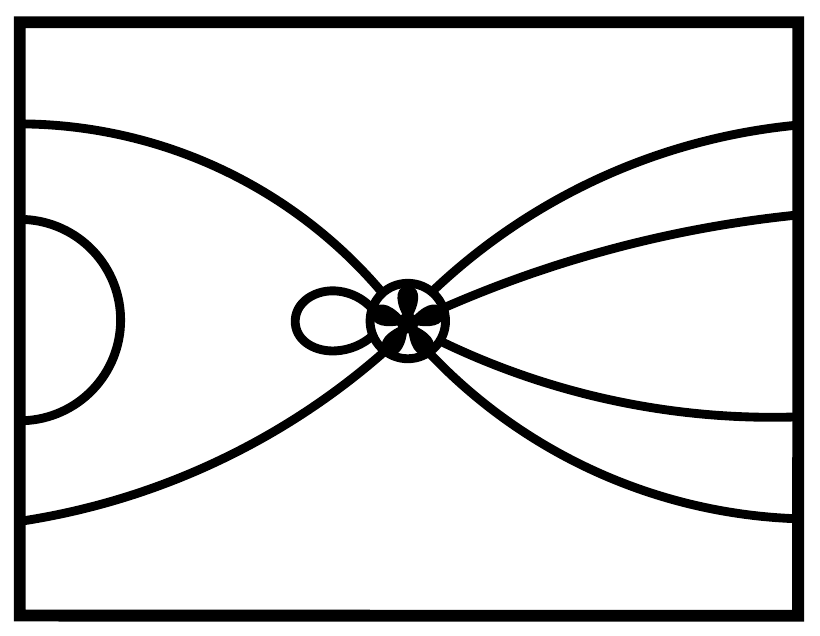}
\put(4, 54){\fontsize{8}{8}$n-i-1$}
\put(62, 54){\fontsize{8}{8}$i-1$}
\put(46, 5){\fontsize{8}{8}$n-i-1$}
\put(4, 5){\fontsize{8}{8}$i-1$}
\end{overpic}}} = \vcenter{\hbox{\begin{overpic}[scale=.35]{Mbfn7.pdf}
\put(4, 54){\fontsize{8}{8}$n-i-1$}
\put(62, 54){\fontsize{8}{8}$i-1$}
\put(46, 5){\fontsize{8}{8}$n-i-1$}
\put(4, 5){\fontsize{8}{8}$i-1$}
\end{overpic}}} \in  (Mb_n)_{n-2}.
\end{equation}
\end{proof}

\begin{corollary}\label{JONESWENZL:antipodal} Sliding $f_n$ through the crosscap is achieved by the following equations. 
\begin{eqnarray}
\vcenter{\hbox{
\begin{overpic}[scale = .35]{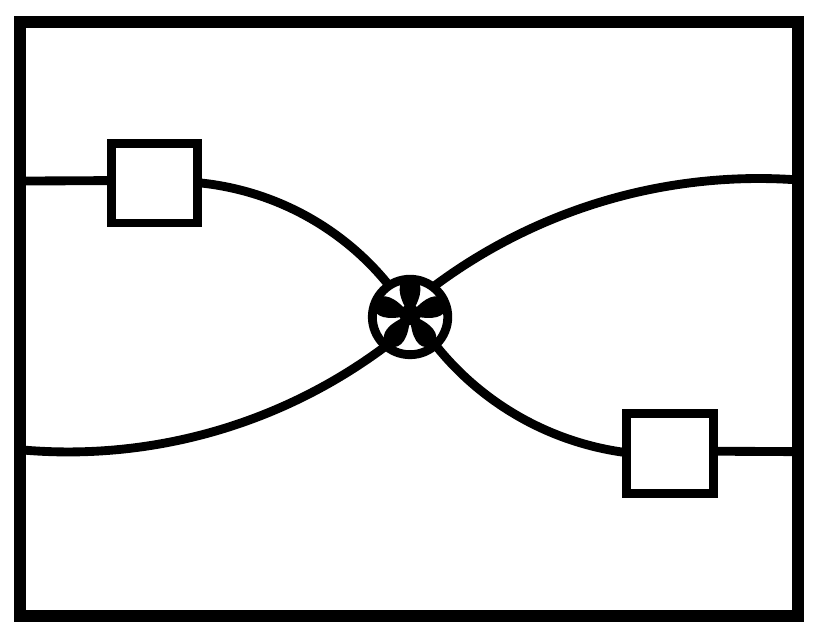}
\put(4, 49){\fontsize{8}{8}$n$}
\put(74, 21){\fontsize{8}{8}$n$}
\end{overpic} }}  &=&
 \vcenter{\hbox{
\begin{overpic}[scale = .35]{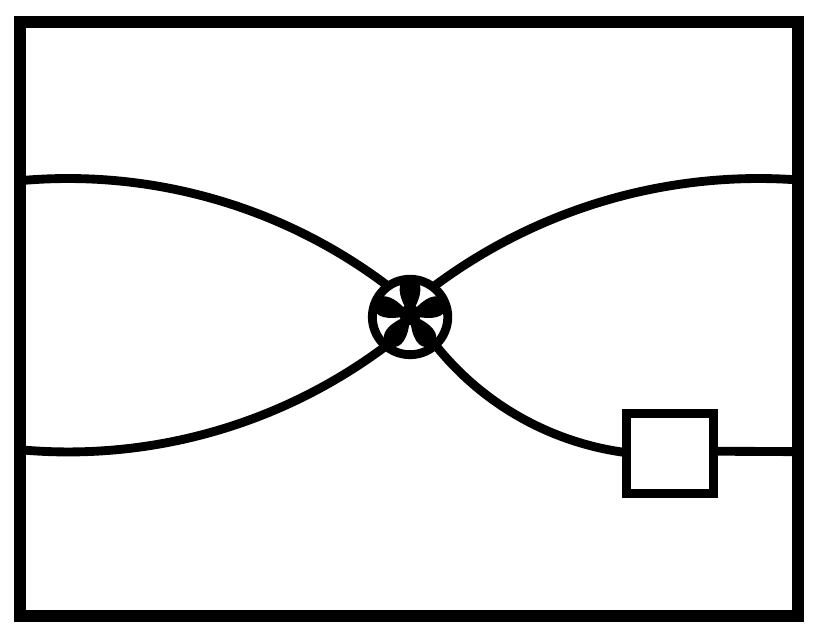}
\put(4, 49){\fontsize{8}{8}$n$}
\put(74, 21){\fontsize{8}{8}$n$}
\end{overpic} }} \label{mob:fnslide1}\\
&=& \vcenter{\hbox{
\begin{overpic}[scale = .35]{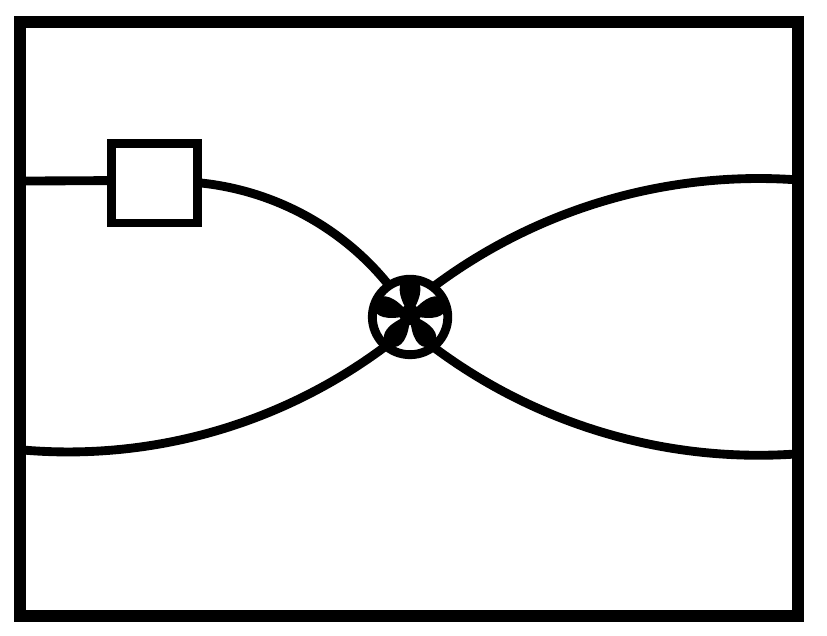}
\put(4, 49){\fontsize{8}{8}$n$}
\put(74, 21){\fontsize{8}{8}$n$}
\end{overpic} }}. \label{mob:fnslide2}
\end{eqnarray}
\end{corollary}

\begin{proof}
By Lemma \ref{JONESWENZL:thmfnbasis}(a) and Lemma \ref{Lemma:eipull}, all of the $e_i$'s coming from the left $f_n$ from left hand side of Equation \ref{mob:fnslide1} can be pulled through the crosscap. Furthermore, for each $e_i$ this action results in a turn back on the second $f_n$. Therefore, we obtain our desired result after applying Lemma \ref{JONESWENZL:thmfnbasis}(b). Equation \ref{mob:fnslide2} is obtained similarly.
\end{proof}

\section{Traces of Jones-Wenzl idempotents in the KBSM of twisted I-bundle of the M\"obius band}

We will now present the three direct corollaries to Lemma \ref{JONESWENZL:lemmacircle} that are obtained from closing $f_b$ through the crosscap.
\begin{corollary}
$$\vcenter{\hbox{
\begin{overpic}[scale = .3]{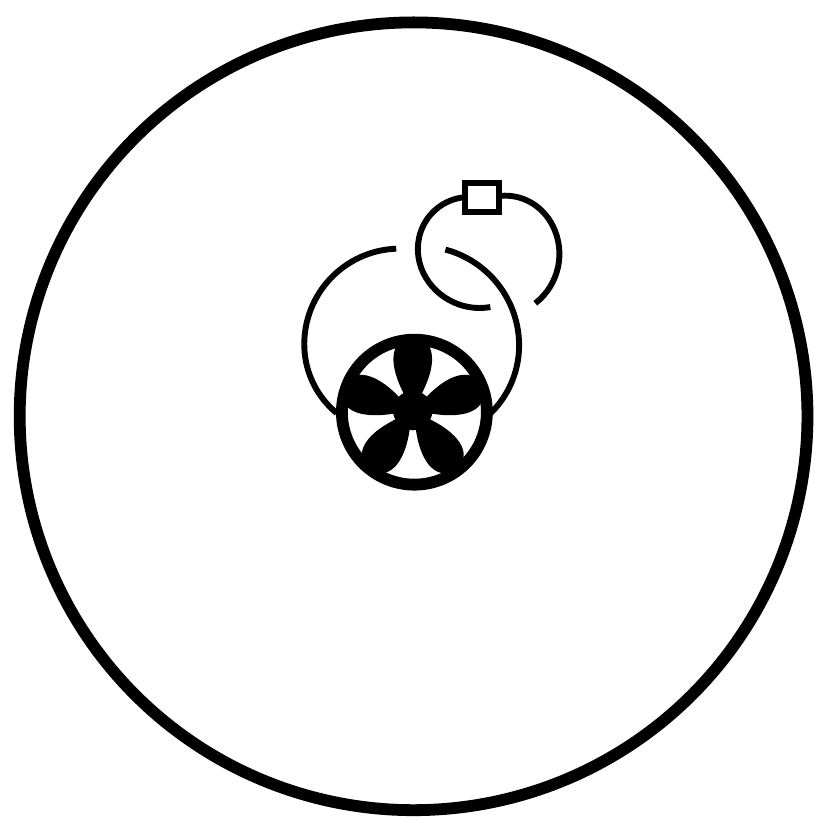}
\put(45, 57){\fontsize{8}{8}$n$}
\end{overpic} }} = \frac{x}{\Delta_1}(-1)^{n}T_{n+1}(d) \Delta_n ,$$
where $T_{k}(d) = (-1)^k(A^{2k}+A^{-2k})$ is the $k^{th}$ Chebyshev polynomial of the first kind and $d=-A^2-A^{-2}$.
\end{corollary}
\begin{proof}
Let $a=n$ and $b=1$ in Lemma \ref{JONESWENZL:lemmacircle}. Then by closing $f_b$ through the crosscap we have
$$ \vcenter{\hbox{
\begin{overpic}[scale = .3]{HopfdecMB1.pdf}
\put(35, 50){\fontsize{6}{6}$n$}
\end{overpic} }} = \frac{(-1)^n (A^{4(n+1)}-A^{-4(n+1)})}{A^4-A^{-4}}\vcenter{\hbox{
\begin{overpic}[scale = .3]{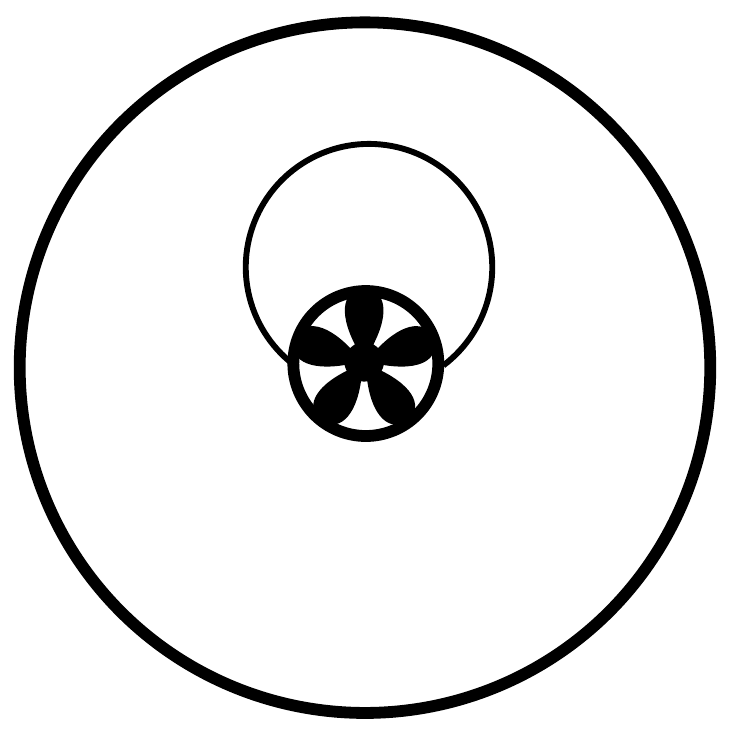} 
\end{overpic} }}.$$
After simplification we have our desired result, where $x$ denotes the simple closed curve that intersects the crosscap once. 
\end{proof}

\begin{corollary}
$$\vcenter{\hbox{
\begin{overpic}[scale = .3]{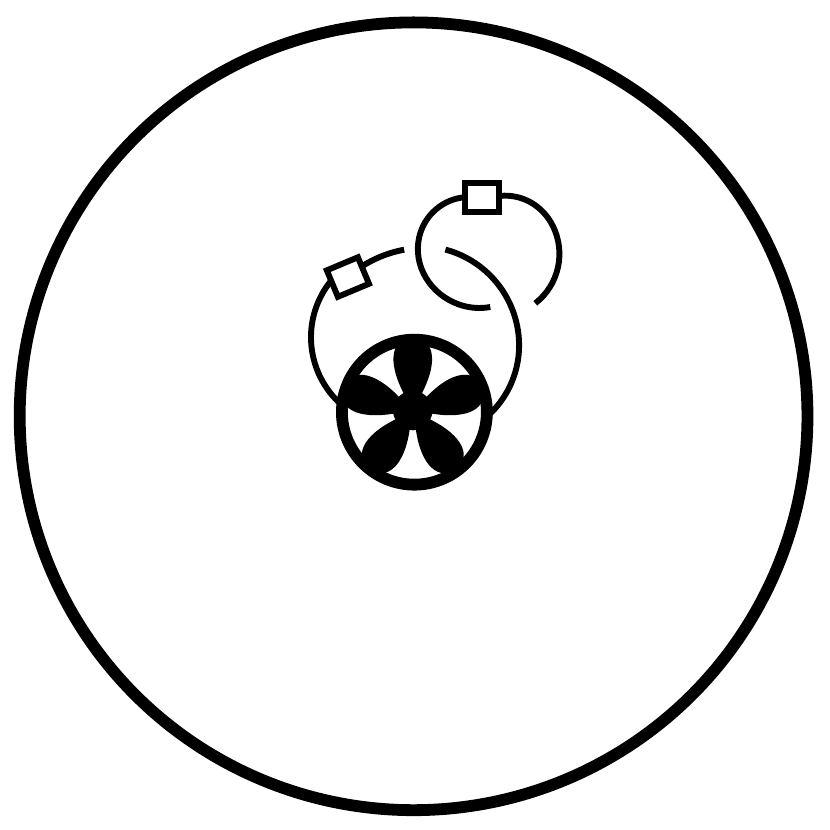}
\put(45, 57){\fontsize{8}{8}$n$}
\put(17, 40){\fontsize{8}{8}$m$}
\end{overpic} }} = \frac{(-1)^n(A^{2(n+1)(m+1)}-A^{-2(n+1)(m+1)})}{A^{2(n+1)}-A^{-2(n+1)}}\vcenter{\hbox{
\begin{overpic}[scale = .3]{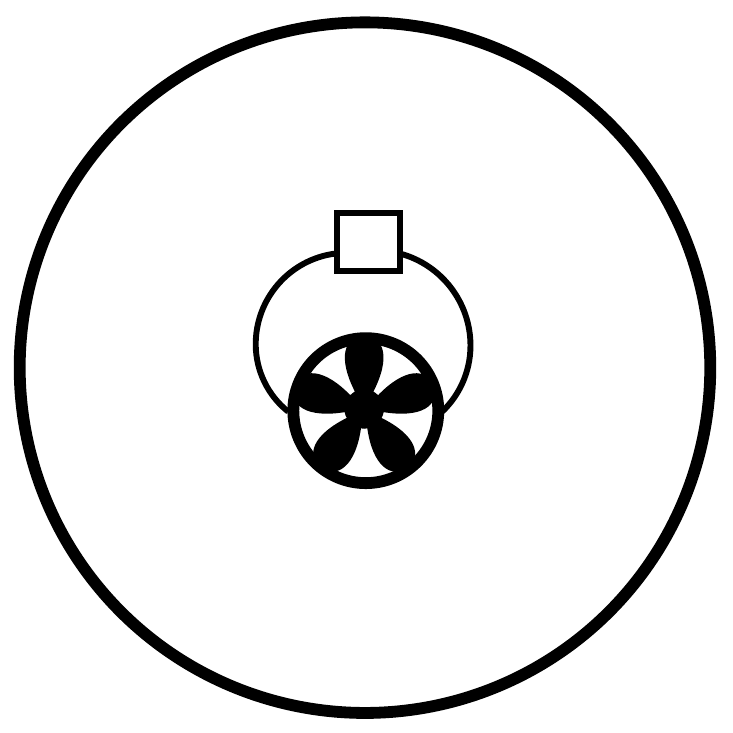}
\put(17, 40){\fontsize{8}{8}$m$}
\end{overpic} }}.$$
\end{corollary}

\begin{proof}
This is obtained from directly applying Lemma \ref{JONESWENZL:lemmacircle} where $a=n$, $b=m$, and $f_b$ is closed through the crosscap.
\end{proof}
\begin{corollary} Let $d=-A^2-A^{-2}$, then
$$\vcenter{\hbox{
\begin{overpic}[scale = .3]{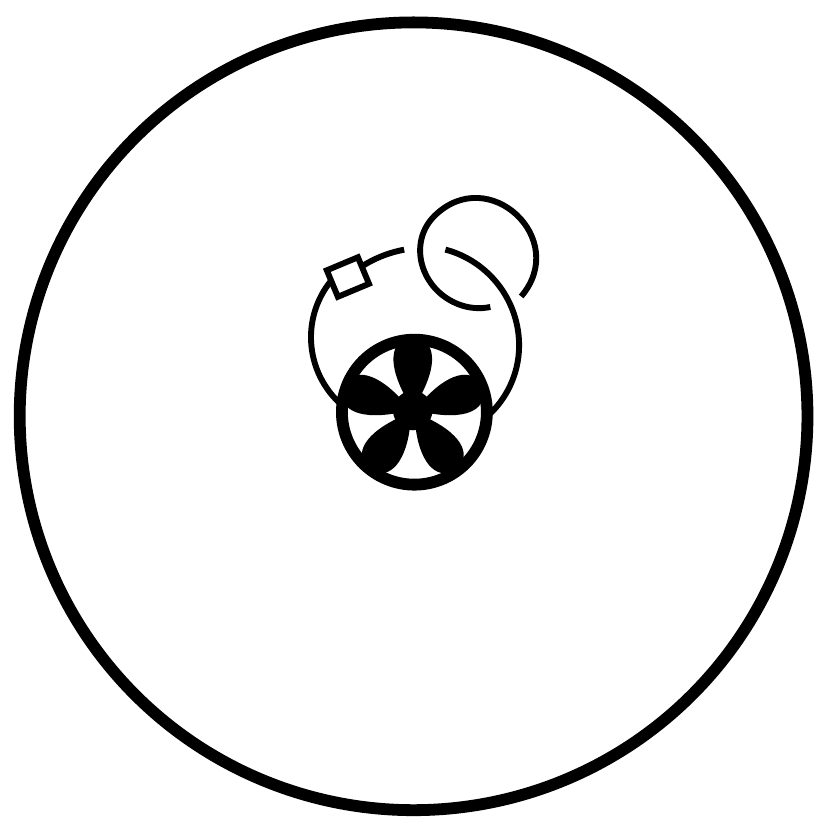}
\put(45, 57){\fontsize{8}{8}$m$}
\put(17, 40){\fontsize{8}{8}$n$}
\end{overpic} }} = ((-1)^{n}T_{n+1}(d))^m \vcenter{\hbox{
\begin{overpic}[scale = .3]{HopfMb9.pdf}
\put(17, 40){\fontsize{8}{8}$n$}
\end{overpic} }}.$$
\end{corollary}

\begin{proof}
We may remove one of the $m$ meridional curves by applying Lemma \ref{JONESWENZL:lemmacircle} where $a=1$ and $b=n$ and closing $f_b$ through the crosscap. After simplification we have
$$ \vcenter{\hbox{
\begin{overpic}[scale = .3]{HopfdecMB3.pdf}
\put(45, 57){\fontsize{8}{8}$m$}
\put(17, 40){\fontsize{8}{8}$n$}
\end{overpic} }}  =  ((-1)^{n}T_{n+1}(d)) \vcenter{\hbox{
\begin{overpic}[scale = .3]{HopfdecMB3.pdf}
\put(34, 57){\fontsize{8}{8}$m-1$}
\put(17, 40){\fontsize{8}{8}$n$}
\end{overpic} }} .$$
We obtain our desired result after repeating this argument $m-1$ more times. 
\end{proof}

The following corollary is obtained from Corollary \ref{Cor:chebyshev} by gluing a crosscap to the inner boundary of the annulus.
\begin{corollary}\label{coro:TraceMBAnn} Let $z$ denote the homotopically non-trivial curve in the M\"obius band that does not intersect the crosscap, then
$$\vcenter{\hbox{
\begin{overpic}[scale = .3]{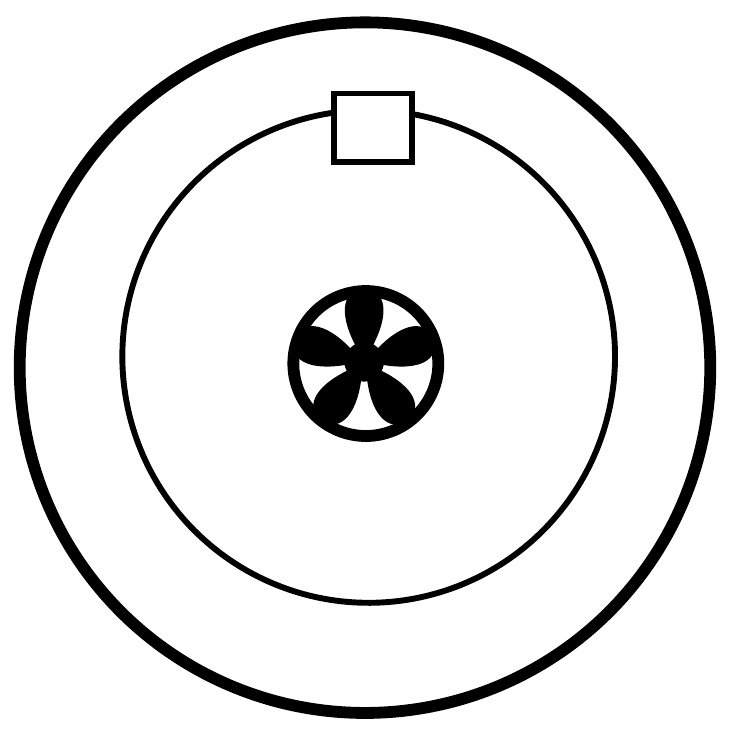}
\put(29, 13){\fontsize{8}{8}$n$}
\end{overpic} }}   = S_n(z). $$
\end{corollary}

\begin{corollary}\label{coro:tracefm}
\begin{eqnarray*}
\vcenter{\hbox{
\begin{overpic}[scale = .45]{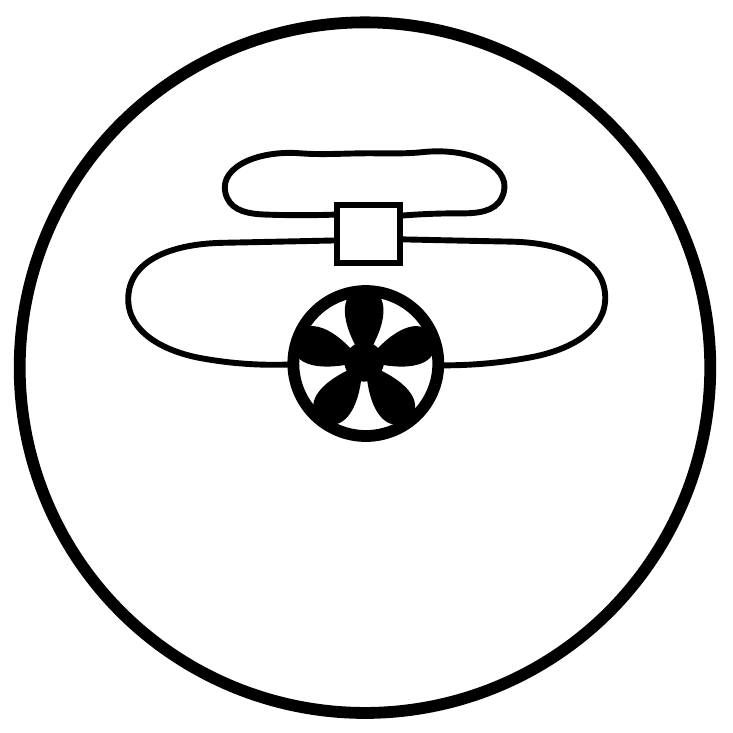}
\put(43, 78){$m$}
\put(60, 50){$n$}
\end{overpic} }} &=& \frac{\Delta_{m+n}}{\Delta_{n}}\vcenter{\hbox{
\begin{overpic}[scale = .45]{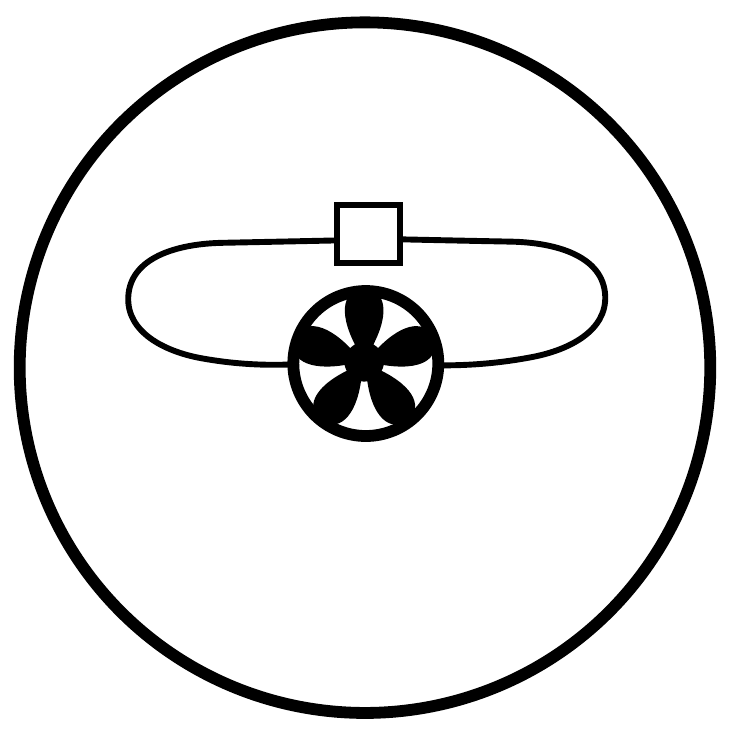}
\put(59, 67){$n$}
\end{overpic} }}.
\end{eqnarray*}
\end{corollary}

\begin{proof}
Apply Corollary \ref{JONESWENZL:corollarytrace}  $m$ times when closing $n$ strands from $f_{n+m}$ through the crosscap then closing the rest away from the crosscap. 
\end{proof}

Let $\mathit{tr}_{Mb_1}(f_n)$ be obtained from closing one arc from $f_n$ through the crosscap and closing the rest of the arcs in such a way that it surrounds the crosscap, as shown in Figure \ref{Mobius:trace}. Then 
Lemma \ref{JONESWENZL:corollarytrace} can no longer be directly applied. Instead we start with applying Wenzl's recursion formula, Theorem \ref{JONESWENZL:wenzlrecursion}, to obtain a recursive formula for $\mathit{tr}_{\mathit{Mb}_1}(f_n)$.

\begin{figure}[ht] 
\centering
$\vcenter{\hbox{\begin{overpic}[scale=.4]{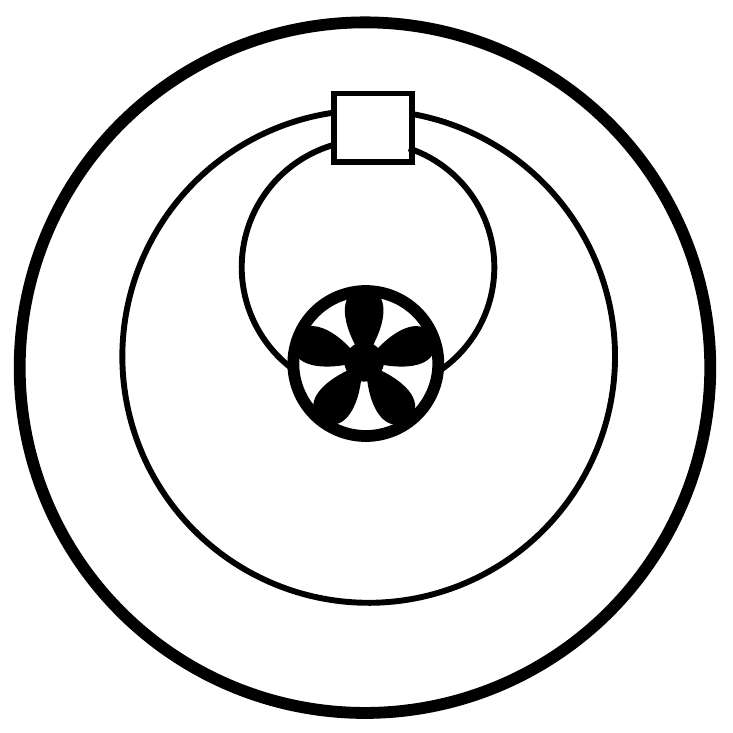}
\put(35, 9){\fontsize{6}{6}$n-1$}
\end{overpic}}}$
\caption{Illustration of $\mathit{tr}_{Mb_1}(f_n)$.}
\label{Mobius:trace}
\end{figure}

\begin{lemma}\label{lemma:Mb1}
$$ \mathit{tr}_{\mathit{Mb}_1}(f_n) = xS_{n-1}(z) - \frac{\Delta_{n-2}}{\Delta_{n-1}} \mathit{tr}_{\mathit{Mb}_1}(f_{n-1}).$$

\end{lemma}
\begin{proof}
By applying Wenzl's formula to the $x$-curve we have the following recursive formula.
\begin{eqnarray*}
\vcenter{\hbox{
\begin{overpic}[scale = .4]{HopfMb3.pdf}
\put(35, 9){\fontsize{6}{6}$n-1$}
\end{overpic} }} &=& \vcenter{\hbox{
\begin{overpic}[scale = .4]{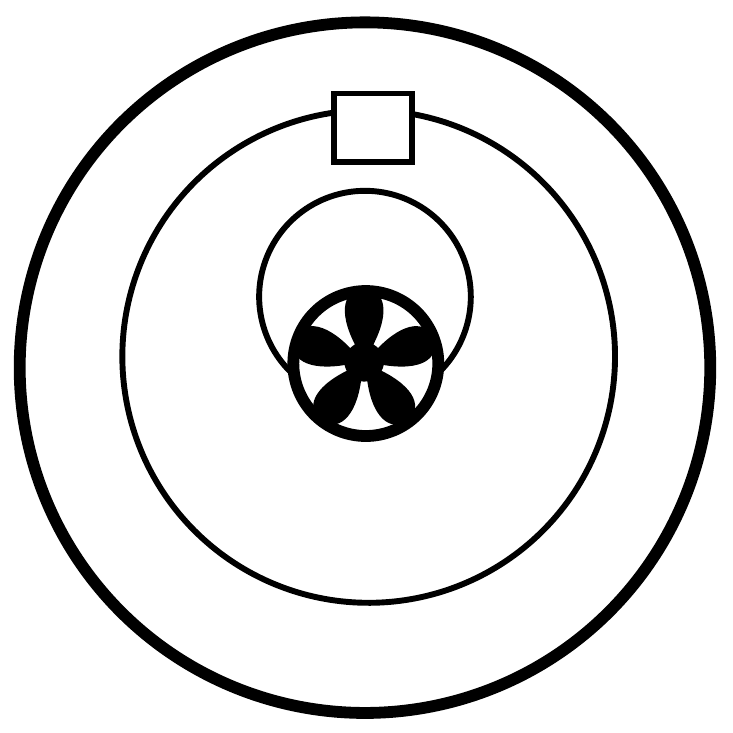}
\put(35, 8){\fontsize{6}{6}$n-1$}
\end{overpic} }} -\frac{\Delta_{n-2}}{\Delta_{n-1}} \vcenter{\hbox{
\begin{overpic}[scale = .4]{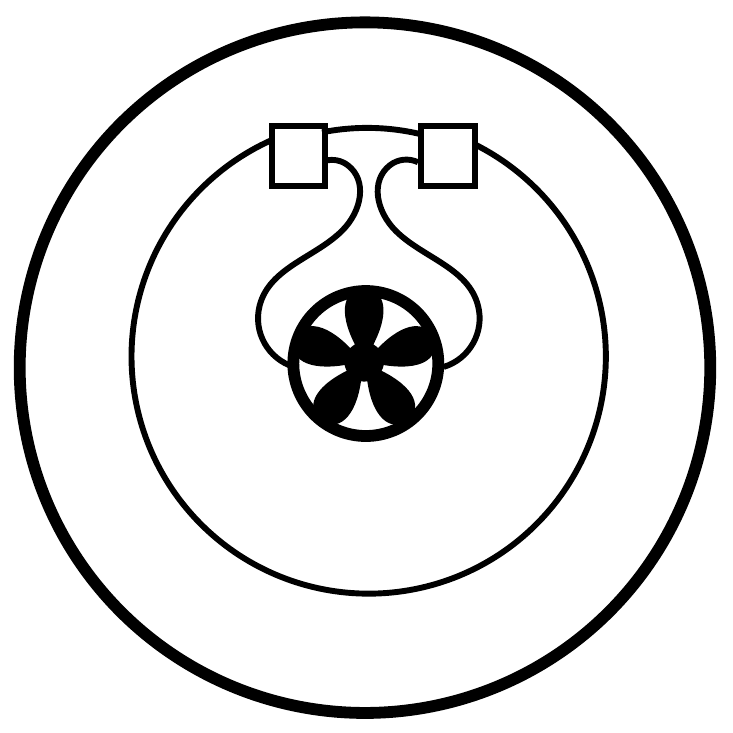}
\put(35, 9){\fontsize{6}{6}$n-1$}
\end{overpic} }}.
\end{eqnarray*}
By Corollary \ref{coro:TraceMBAnn} and Lemma \ref{JONESWENZL:thmfnbasis}(c),

\begin{eqnarray*}
\vcenter{\hbox{
\begin{overpic}[scale = .4]{HopfMb3.pdf}
\put(35, 9){\fontsize{6}{6}$n-1$}
\end{overpic} }} &=& x S_{n-1}(z) -\frac{\Delta_{n-2}}{\Delta_{n-1}} \vcenter{\hbox{
\begin{overpic}[angle=180, scale = .4]{HopfMb3.pdf}
\put(33, 73){\fontsize{6}{6}$n-2$}
\end{overpic} }}.
\end{eqnarray*}
\end{proof}

\begin{proposition}
$$\mathit{tr}_{Mb_1}(f_n) = \frac{x}{\Delta_{n-1}}\sum_{k=0}^{n-1} (-1)^{n-1+k} S_{k}(z)\Delta_{k}. $$
\end{proposition}

\begin{proof}
The base case is trivial, 
$$\mathit{tr}_{Mb_1}(f_1) = \vcenter{\hbox{
\begin{overpic}[scale = .3]{HopfMb2.pdf} 
\end{overpic} }}=x .$$

Suppose $\mathit{tr}_{Mb_1}(f_{n-1}) = \frac{x}{\Delta_{n-2}}\sum_{k=0}^{n-2} (-1)^{n-2+k} S_{k}(z)\Delta_{k}.$ Then by Lemma \ref{lemma:Mb1}, 
\begin{eqnarray*}
  \mathit{tr}_{\mathit{Mb}_1}(f_n) &=& xS_{n-1}(z) - \frac{\Delta_{n-2}}{\Delta_{n-1}} \mathit{tr}_{\mathit{Mb}_1}(f_{n-1}) \\
&=&  xS_{n-1}(z) - \frac{x}{\Delta_{n-1}} \sum_{k=0}^{n-2} (-1)^{n-2+k} S_{k}(z)\Delta_{k} \\
&=& \frac{x}{\Delta_{n-1}}  S_{n-1}(z) \Delta_{n-1} + \frac{x}{\Delta_{n-1}} \sum_{k=0}^{n-2} (-1)^{n-1+k} S_{k}(z)\Delta_{k}.
\end{eqnarray*}
\end{proof}

As discussed in Section \ref{MBcrossingless}, the number of times an arc intersects the crosscap can be reduced by an even number. Therefore, we expect a different formula when closing two arcs, versus one arc, from $f_n$ through the crosscap. \\

Let $\mathit{tr}_{Mb_2}(f_n)$ be obtained from closing two arcs from $f_n$ through the crosscap and closing the rest of the arcs in such a way that it surrounds the crosscap, similar to Figure \ref{Mobius:trace}. Then, as expected, by applying Wenzl's formula we see a different recursive formula than that of Lemma \ref{lemma:Mb1}. 

\begin{lemma}\label{Lemma:Mb2}  For $n \geq 2$,
    \begin{equation*}
    tr_{Mb_2}(f_n) = S_{n-1}(z)-S_{n-2}(z)+ \frac{\Delta_{n-3}}{\Delta_{n-1}} tr_{Mb_{2}}
(f_{n-1}).\end{equation*}
\end{lemma}

\begin{proof}
    By applying Wenzl's formula to the inner most curve intersecting the crosscap, then by removing two intersection points from the first term in the sum and by applying Corollary \ref{coro:TraceMBAnn} we have the following formula.
\begin{eqnarray*}
\vcenter{\hbox{
\begin{overpic}[scale = .4]{HopfMb3.pdf}
\put(35, 9){\fontsize{6}{6}$n-2$}
\put(22, 50){\fontsize{6}{6}$2$}
\end{overpic} }} &=& S_{n-1}(z) -\frac{\Delta_{n-2}}{\Delta_{n-1}} \vcenter{\hbox{
\begin{overpic}[scale = .4]{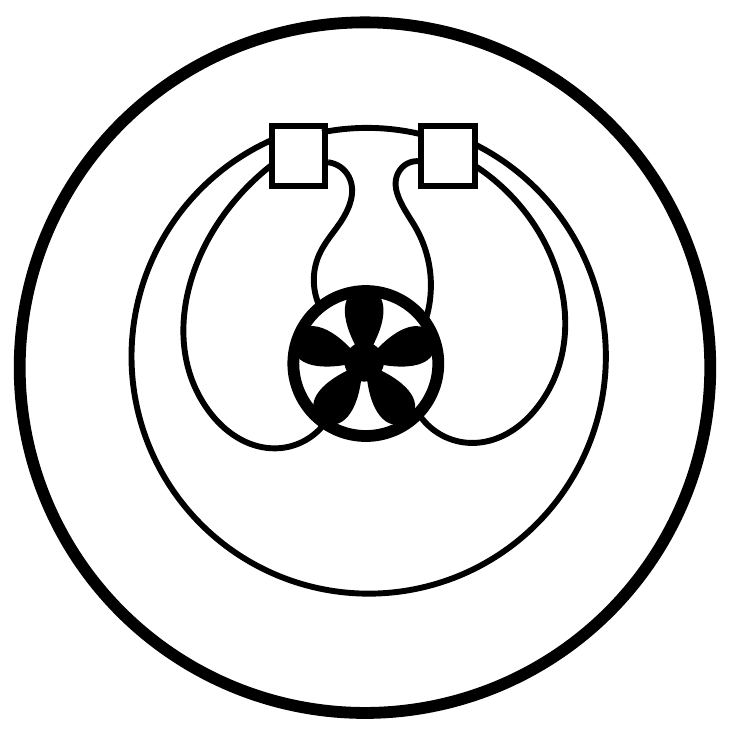}
\put(35, 9){\fontsize{6}{6}$n-2$}
\put(35, 73){\fontsize{6}{6}$n-2$}
\end{overpic} }}.
\end{eqnarray*}
We obtain the following equation by applying Wenzl's formula on the last term of the sum and then by applying Lemma \ref{JONESWENZL:thmfnbasis} and removing two intersection points from the third term.

\begin{eqnarray*}
\vcenter{\hbox{
\begin{overpic}[scale = .4]{HopfMb3.pdf}
\put(35, 9){\fontsize{6}{6}$n-2$}
\put(22, 50){\fontsize{6}{6}$2$}
\end{overpic} }} &=& S_{n-1}(z) \\
&&-\frac{\Delta_{n-2}}{\Delta_{n-1}} \left( \vcenter{\hbox{
\begin{overpic}[scale = .4]{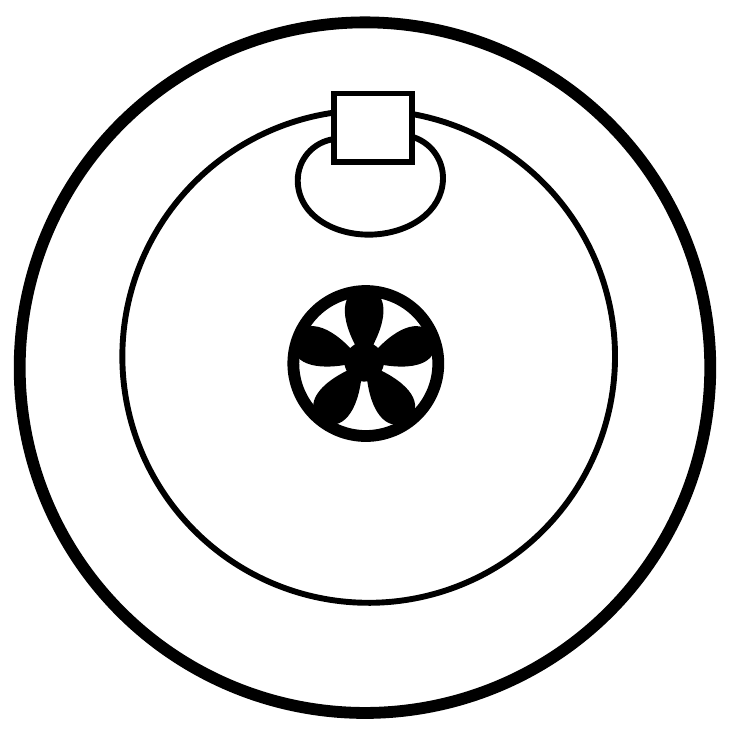}
\put(35, 8){\fontsize{6}{6}$n-2$}
\end{overpic} }} -\frac{\Delta_{n-3}}{\Delta_{n-2}} \vcenter{\hbox{
\begin{overpic}[scale = .4]{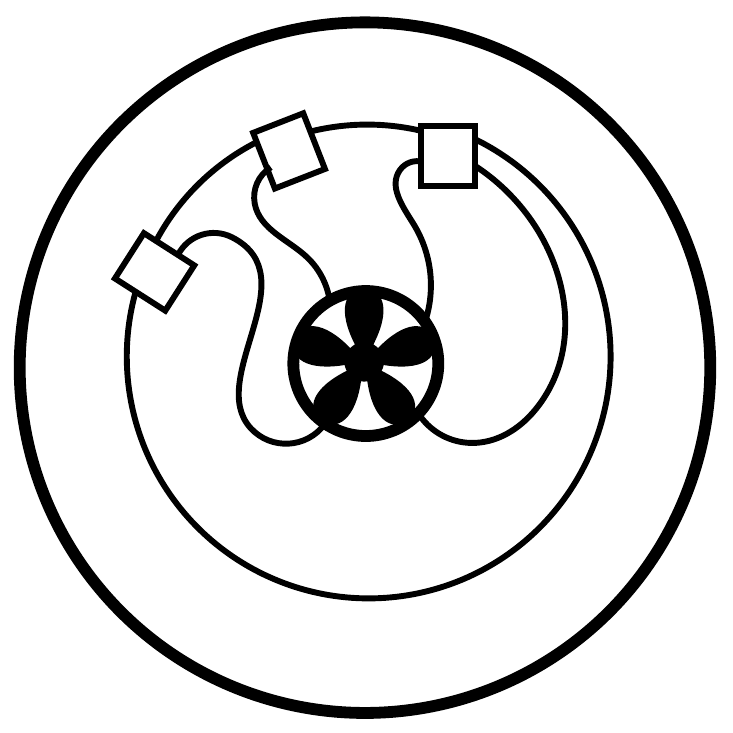}
\put(35, 8){\fontsize{6}{6}$n-2$}
\put(34, 74){\fontsize{6}{6}$n-2$}
\put(12, 59){\rotatebox{40}{\fontsize{6}{6}$n-3$}}
\end{overpic} }}\right).
\end{eqnarray*}
We obtain our desired result after applying Corollaries \ref{coro:tracefm} and \ref{coro:TraceMBAnn} and Lemma \ref{JONESWENZL:thmfnbasis}(d).
\end{proof}

From Lemma \ref{Lemma:Mb2} we obtain a formula for $\mathit{tr}_{Mb_2}(f_n)$ and it reveals that $\mathit{tr}_{\mathit{Mb}_2}(f_n)$ is not a factor of $x$.

\begin{proposition}
    For $n \geq 2$,
$$ \mathit{tr}_{\mathit{Mb}_2}(f_n) =  \sum_{i=0}^{n-2} \frac{\Delta_{i+1} \Delta_{i}}{\Delta_{n-1}\Delta_{n-2}}(S_{i+1}(z)-S_{i}(z)). $$
\end{proposition}

\begin{proof}
        Since $\mathit{tr}_{\mathit{Mb}_2}(f_2)$ is just $f_2$ closed through the crosscap, then by applying Wenzl's formula and by reducing even intersection points on the crosscap, we have
    \begin{equation}\label{Eqn:Mb2}
        \mathit{tr}_{\mathit{Mb}_2}(f_2) = \vcenter{\hbox{
\begin{overpic}[scale = .3]{HopfMb9.pdf}
\put(40, 40){\fontsize{7}{7}$2$} 
\end{overpic} }} = S_1(z) - S_0(z).
    \end{equation}
    By Equation \ref{Eqn:Mb2}, the base case holds. \\

    Suppose $\mathit{tr}_{\mathit{Mb}_2}(f_{n-1}) = \frac{1}{\Delta_{n-2}\Delta_{n-3}} \sum_{i=1}^{n-2} \Delta_i \Delta_{i-1} (S_i(z)-S_{i-1}(z))$, then by Lemma \ref{Lemma:Mb2},

    \begin{eqnarray*}
    tr_{Mb_2}(f_n) &=& S_{n-1}(z)-S_{n-2}(z)+ \frac{\Delta_{n-3}}{\Delta_{n-1}} tr_{Mb_{2}}
(f_{n-1}) \\
&=& S_{n-1}(z)-S_{n-2}(z) \\
&&+ \frac{1}{\Delta_{n-1}\Delta_{n-2}} \sum_{i=1}^{n-2} \Delta_i \Delta_{i-1} (S_i(z)-S_{i-1}(z)).
    \end{eqnarray*}
\end{proof}

\section{Future Directions}
By the classification of unorientable surfaces we may extend the results given in this paper to the KBSM of the twisted $I$-bundle over unorientable surfaces as well as Temperley-Lieb algebra-modules over unorientable surfaces with non-empty boundary. Future work could explore calculating analog formulas to the annular case and extending the work to unorientable surfaces then applying it to evaluating knots decorated with $f_n$ in the twisted $I$-bundle over unorientable surfaces. Since the calculations of the Jones-Wenzl idempotent of the KBSM of the twisted $I$-bundle over the M\"obius band preserve the $I$-bundle structure, then there is evidence that these calculations are giving extra information about the $I$-bundle structure and that extending to different fiber structures on the solid torus in a similar manner could uncover a different method to studying fibered manifolds. 



\begin{thebibliography}{99999999}
\bibitem[BIMP]{BIMP} R. P. Bakshi, D. Ibarra, S. Mukherjee, J. H. Przytycki, A generalization of the Gram determinant of type A, \textit{Topology Appl.} 295 (2021), Paper No. 107663, 15 pp.
e-print: \href{https://arxiv.org/abs/1905.07834}{arXiv:1905.07834} [math.GT].
\bibitem[Bax]{Bax2} R. J. Baxter, Exactly solved models in statistical mechanics. \textit{Academic Press, Inc.}, London (1982).
\bibitem[Cai]{Cai} X. Cai, A Gram determinant of Lickorish's bilinear form. \textit{Math. Proc. Cambridge Philos. Soc.} 151 (2011), no. 1, 83–94. 
\href{https://arxiv.org/pdf/1006.1297.pdf}
{arXiv:1006.1297v3} [math.GT].
 \bibitem[Jon]{Jon1} V. F. R. Jones, Index for subfactors. \textit{Invent. Math.} 72, 1983, 1-25.
\bibitem[Kau]{Kau2} L. H. Kauffman,
An invariant of regular isotopy. 
\textit{Trans. Amer. Math. Soc.} 318 (1990), no. 2, 417–471.
\bibitem[Le]{Le}
T. T. Q. Lê, The colored Jones polynomial and the A-polynomial of knots. \textit{Adv. Math.} 207 (2006), no. 2, 782–804. \href{https://arxiv.org/abs/math/0407521}{arXiv:math/0407521} [math.GT].

\bibitem[Lic1]{Lic2} W. B. R. Lickorish, Homeomorphisms of non-orientable two-manifolds,	{\it Proc. Camb. Phil. Soc.}, 59 (1963), pp. 307-317

\bibitem[Lic2]{Lic3} W. B. R. Lickorish, On the homeomorphisms of a non-orientable surface, Proc. Cambridge Philos. Soc. 61 (1965), 61–64.


\bibitem[Lic3]{Lic1} W. B. R.~Lickorish, Invariants for 3-manifolds from the combinatorics of the Jones polynomial, \textit{Pacific Journ. Math.},149(2), 1991, 337-347. 
 

\bibitem[Lic4]{Lictln}  W. B, R.  Lickorish, Three-manifolds and the Temperley-Lieb algebra. \textit{Math. Ann.} 290 (1991), no. 4, 657–670. 

\bibitem[Lic5]{LicCalc} W. B. R. Lickorish,  Calculations with the Temperley-Lieb algebra. \textit{Comment. Math. Helv.} 67 (1992), no. 4, 571–591.

\bibitem[Lic6]{Lic} W. B. R. Lickorish, An introduction to knot theory. (English summary)
{\it Graduate Texts in Mathematics}, 175. Springer-Verlag, New York, 1997. 

\bibitem [Prz]{Prz} J. H. Przytycki, Fundamentals of Kauffman bracket skein modules. {\it Kobe Math. J.}, 16(1), 1999, 45-66. \href{https://arxiv.org/abs/math/9809113}{arXiv:math/9809113} [math.GT].
\bibitem[PBIMW]{PBIMW} J. H. Przytycki, R. P. Bakshi, D. Ibarra, G. Montoya-Vega, D. E. Weeks, Lectures on Knot Theory: An Exploration of Contemporary Topics, \textit{Springer Universitext} (to appear).
\bibitem[TL]{TL}  H. Temperley and E. Lieb, Relations Between the ‘Percolation’ and ‘Colouring’ Problem and Other Graph-Theoretic
Problems Associated with Regular Plane Lattices: Some Exact Results for the ‘Percolation’ Problem, \textit{Proceeds of the
Royal Society of London} 322 (1971), 251 - 280.
\bibitem[TV]{TV} V. G. Turaev, O. Ya. Viro,
State sum invariants of 3-manifolds and quantum 6j-symbols.
\textit{Topology} 31 (1992), no. 4, 865–902.
\bibitem[Wen]{Wen} H.~Wenzl, On sequences of projections,  \textit{C.R. Math. Rep. Acad. Sci.}, IX, 1987, 5-9.
	\end{thebibliography}
\end{document}